\documentclass[a4paper,UKenglish,cleveref,autoref,thm-restate]{lipics-v2021}
\pdfoutput=1

\usepackage{microtype}
\usepackage{epigraph}

\usepackage{amsmath,amsthm}
\usepackage{pdfrender}
\usepackage{epsfig,psfrag}
\usepackage{graphicx}
\usepackage{graphics}
\usepackage{latexsym}
\usepackage{amssymb}
\usepackage{amsmath}
\usepackage{stmaryrd}
\usepackage{ifthen}
\usepackage{array}
\usepackage{mathrsfs}
\usepackage{multirow}
\usepackage{bbm}

\usepackage{pgf}
\usepackage{pgfplots}
\pgfplotsset{compat=1.16} 
\usepackage{tikz, mathtools, stmaryrd}
\usetikzlibrary{arrows,calc,automata,backgrounds,decorations,intersections,pgfplots.fillbetween,shapes.misc}
\tikzset{every picture/.style={thick,>=angle 60}}
\tikzset{MDPrand/.style={draw,circle,minimum size=11*1.5,inner sep=0}}
\tikzset{MDPcont/.style={draw,rectangle,minimum size=9*1.5,inner sep=0}}
\tikzset{MDPbad/.style={fill=red}}
\definecolor{myorange}{RGB}{255, 128, 0}

\usetikzlibrary{shapes.multipart,shapes.misc}
\usepgflibrary{decorations.pathreplacing}
\usetikzlibrary{arrows,calc,topaths}
\usetikzlibrary{automata,shapes.arrows, backgrounds, fadings,intersections, positioning}
\usetikzlibrary{shadows}

\usetikzlibrary{decorations.pathmorphing}
\usetikzlibrary{decorations.shapes}
\usetikzlibrary{shapes}
\usetikzlibrary{backgrounds}
\usetikzlibrary{fit}
\usetikzlibrary{arrows}
\usetikzlibrary{calc}
\usetikzlibrary{mindmap}
\usetikzlibrary{matrix}

 \usepackage{todonotes}

\usepackage{mathtools}
\usepackage{thmtools}
\usepackage{xcolor}

\usepackage{bbding}

\usepackage{hyperref} 
\hypersetup{colorlinks,linkcolor=blue,citecolor=blue,urlcolor=blue}

\newcommand{\+}[1]{\mathbb{#1}}

\newcommand{\N}{\+{N}}

\newcommand{\x}{\times}

\usepackage{wasysym} 
\newcommand{\rsymbol}{\ocircle}
\newcommand{\zsymbol}{\Box}
\newcommand{\zstates}{\states_\zsymbol}
\newcommand{\rstates}{\states_\rsymbol}
\newcommand{\reachset}{T}

\newcommand{\eqby}[2][=]{\stackrel{\text{{\tiny{#2}}}}{#1}}
\newcommand{\eqdef}{\eqby{def}}

\newcommand{\eps}{\varepsilon}
\newcommand{\step}[2][]{\xrightarrow[#1]{#2}}

\newcommand{\problemx}[3]{
\par\noindent\underline{\sc#1}\par\nobreak\vskip.2\baselineskip
\begingroup\clubpenalty10000\widowpenalty10000
\setbox0\hbox{\bf INPUT:\ }\setbox1\hbox{\bf QUESTION:\ }
\dimen0=\wd0\ifnum\wd1>\dimen0\dimen0=\wd1\fi
\vskip-\parskip\noindent
\hbox to\dimen0{\box0\hfil}\hangindent\dimen0\hangafter1\ignorespaces#2\par
\vskip-\parskip\noindent
\hbox to\dimen0{\box1\hfil}\hangindent\dimen0\hangafter1\ignorespaces#3\par
\endgroup}

\newcommand{\dist}{\mathcal{D}}

\newcommand{\cParity}[1]{{#1}\text{-}\mathtt{Parity}}
\newcommand{\Parity}[1]{\mathtt{Parity}(#1)}

\newcommand{\reach}[1]{\mathtt{Reach}(#1)}

\newcommand{\transient}{\mathtt{Transience}}
\newcommand{\safety}[1]{\mathtt{Safety}(#1)}
\newcommand{\always}{{\sf G}}
\newcommand{\eventually}{{\sf F}}
\renewcommand{\next}{{\sf X}}

\newcommand{\hide}[1]{}

\newcommand{\lrc}[1]{(#1)}
\newcommand{\lrd}[1]{\{#1\}}

\newcommand{\ignore}[1]{}

\newcommand{\nat}{\mathbb N}

\newcommand{\setcomp}[2]{\lrd{{#1}|\;{#2}}}
\newcommand{\tuple}[1]{\lrc{#1}}

\newcommand{\denotationof}[2]{\llbracket #1\rrbracket^{#2}}

\newcommand{\mdp}{{\mathcal M}}

\newcommand{\mdptuple}{\tuple{\states,\zstates,\rstates,\transition,\probp}}

\newcommand{\states}{S}
\newcommand{\initstates}{I}

\newcommand{\stateset}{Q}
\newcommand{\state}{s}

\newcommand{\transition}{{\longrightarrow}}

\newcommand{\probp}{P}

\newcommand{\smallparg}[1]{{\bf #1}}

\newcommand{\complementof}[1]{\overline{#1}}

\newcommand{\play}{\rho}
\newcommand{\playset}{{\mathfrak R}}

\newcommand{\partialplay}{\rho}

\newcommand{\zstrat}{\sigma}

\newcommand{\zallstrats}[1]{\zstratset_{{#1}}}

\newcommand{\zstratset}{\Sigma}

\newcommand{\memory}{{\sf M}}
\newcommand{\updatefun}{u}
\newcommand{\memconf}{{\sf m}}

\newcommand{\expectval}{{\mathcal E}}
\newcommand{\probm}{{\mathcal P}}
\newcommand{\expectation}[1][]{ \expectval_{#1}}

\newcommand{\formula}{{\varphi}}

\newcommand{\colorset}[3]{[#1]^{\coloring#2#3}}
\newcommand{\coloring}{{\mathit{C}ol}}
\newcommand{\colorof}[1]{\coloring\lrc{{#1}}}

\newcommand{\cset}{{\mathcal C}}

\newcommand{\valueof}[2]{{\mathtt{val}_{#1}(#2)}}
\newcommand{\optval}{v^*}

\newcommand{\constraint}{\rhd}

\mathchardef\mhyphen="2D 

\newcommand{\bubble}[2]{{\sf bubble}_{#1}(#2)}
\newcommand{\firstinset}[1]{{\sf firstin}(#1)}

\newcommand{\reset}{\mbox{\upshape\texttt{B\"uchi}}}
\newcommand{\setf}{\eventually}
\newcommand{\setfb}[2]{\setf^{{#1}}(#2)}

\newcommand{\even}{{\mathit even}}

\newcommand{\cost}{{\mathtt{cost}}}
\newcommand{\costRV}{{\mathtt{Cost}}}

\newcommand{\cyl}{\mathfrak C}
\newcommand{\classcyl}{\mathcal{C}}
\newcommand{\classmon}{\mathcal{Q}}

\newcommand{\pmdp}{\mdp_{*}}
\newcommand{\pstates}{\states_{*}}
\newcommand{\pzstates}{\states_{*\zsymbol}}
\newcommand{\prstates}{\states_{*\rsymbol}}
\newcommand{\ptransition}{\transition_{*}}
\newcommand{\pprobp}{\probp_{*}}

\title{Transience in Countable MDPs}

\author{Stefan Kiefer}{Department of Computer Science, University of Oxford, UK}{}{}{}
\author{Richard Mayr}{School of Informatics, University of Edinburgh, UK}{}{}{}
\author{Mahsa Shirmohammadi}{Universit\'e de Paris, CNRS, IRIF, F-75013 Paris, France}{}{}{}
\author{Patrick Totzke}{Department of Computer Science, University of Liverpool, UK}{}{}{}
\authorrunning{S.~Kiefer, R.~Mayr, M.~Shirmohammadi, P.~Totzke}
\Copyright{Stefan Kiefer, Richard Mayr, Mahsa Shirmohammadi, Patrick Totzke}
\ccsdesc[200]{Theory of computation~Random walks and Markov chains}
\ccsdesc[100]{Mathematics of computing~Probability and statistics}
\keywords{Markov decision processes, Parity, Transience}

\relatedversion{This is the full version of a CONCUR 2021 paper \cite{KMST:CONCUR2021}.}

\supplement{}

\funding{}

\acknowledgements{}

\hideLIPIcs  
\nolinenumbers 

\EventEditors{Serge Haddad and Daniele Varacca}
\EventNoEds{2}
\EventLongTitle{32nd International Conference on Concurrency Theory (CONCUR 2021)}
\EventShortTitle{CONCUR 2021}
\EventAcronym{CONCUR}
\EventYear{2021}
\EventDate{August 23--27, 2021}
\EventLocation{Virtual Conference}
\EventLogo{}
\SeriesVolume{203}
\ArticleNo{1}

\begin{document}

\maketitle

\begin{abstract}
The $\transient$ objective is not to visit any state infinitely often.
While this is not possible in any finite Markov Decision Process (MDP), it can be satisfied in
countably infinite ones, e.g., if the transition graph is acyclic.

We prove the following fundamental properties of $\transient$
in countably infinite MDPs.
\begin{enumerate}
\item
  There exist uniformly $\eps$-optimal MD strategies  (memoryless
  deterministic) for $\transient$, even in infinitely branching MDPs.
\item
  Optimal strategies for $\transient$ need not exist,
  even if the MDP is finitely branching.
  However, if an optimal strategy exists then there is also an optimal MD strategy.
\item
If an MDP is universally transient (i.e., almost surely transient under
all strategies) then many other objectives have a lower strategy complexity
than in general MDPs. E.g., $\eps$-optimal strategies for Safety and co-B\"uchi and
optimal strategies for $\cParity{\{0,1,2\}}$ (where they exist) can be chosen
MD, even if the MDP is infinitely branching.
\end{enumerate}
\end{abstract}

\section{Introduction}\label{sec:intro}
\epigraph{
Those who cannot remember the past are condemned to repeat it.
}{\textit{George Santayana (1905) \cite{Satayana:wiki}}}
The famous aphorism above has often been cited (with small variations), e.g.,
by Winston Churchill in a 1948 speech to the House of Commons, and carved into
several monuments all over the world \cite{Satayana:wiki}.

We prove that the aphorism is false.
In fact, even those who cannot remember anything at all are \emph{not}
condemned to repeat the past.
With the right strategy they can avoid repeating the past equally well
as everyone else.
More formally, playing for $\transient$ does not require any
memory.
We show that there always exist $\eps$-optimal
memoryless deterministic strategies for $\transient$,
and if optimal strategies exist then there also exist
optimal memoryless deterministic strategies.\footnote{Our result applies to MDPs (also called games against nature).
It is an open question whether it generalizes to countable stochastic 2-player games.
(However, it is easy to see that the adversary needs infinite memory in general, even if
the player is passive \cite{KMST:ICALP2019,KMSTW2020}.)
}

{\bf\noindent Background.}
We study Markov decision processes (MDPs), a standard model for dynamic systems that
exhibit both stochastic and controlled behavior \cite{Puterman:book}.
MDPs play a prominent role in many domains, e.g., artificial intelligence and machine learning~\cite{sutton2018reinforcement,sigaud2013markov}, control theory~\cite{blondel2000survey,NIPS2004_2569}, operations research and finance~\cite{Sudderth:2020,Hill-Pestien:1987,bauerle2011finance,schal2002markov}, and formal verification~\cite{Flesch:JOTA2020,Sudderth:2020,FPS:2018,ModCheckHB18,ModCheckPrinciples08,chatterjee2012survey}.

An MDP is a directed graph where states are either random or controlled.
Its observed behavior is described by runs, which are infinite paths that are, in part, determined by the choices of a controller.
If the current state is random then the next state is chosen according to a fixed probability distribution.
Otherwise, if the current state is controlled, the controller can choose a distribution over all possible successor states.
By fixing a strategy for the controller (and initial state), one obtains a probability space
of runs of the MDP. The goal of the controller is to optimize the expected value of
some objective function on the runs.

The \emph{strategy complexity} of a given objective characterizes
the type of strategy necessary to achieve an optimal (resp.\ $\eps$-optimal)
value for the objective.
General strategies can take
the whole history of the run into account (history-dependent; (H)),
while others use
only bounded information about it
(finite memory; (F))
or base decisions only on the current state (memoryless; (M)).
Moreover, the strategy type depends on whether the controller can
randomize (R) or is limited to deterministic choices~(D).
The simplest type, MD, refers to memoryless deterministic strategies.

\smallskip
{\bf\noindent Acyclicity and Transience.}
An MDP is called acyclic iff its transition graph is acyclic.
While finite MDPs cannot be acyclic (unless they have deadlocks),
countable MDPs can.
In acyclic countable MDPs, the
strategy complexity of B\"uchi/Parity objectives is lower 
than in the general case:
$\eps$-optimal strategies for B\"uchi/Parity objectives require only one bit
of memory in acyclic MDPs, while they require infinite memory
(an unbounded step-counter, plus one bit) in general countable
MDPs \cite{KMST:ICALP2019,KMST2020c}.

The concept of \emph{transience} can be seen as a generalization of
acyclicity. 
In a Markov chain, a state $\state$ is called \emph{transient} iff the probability of
returning from $\state$ to $\state$ is $<1$ (otherwise the state is called
recurrent).
This means that a transient state is almost surely visited only finitely
often.
The concept of transient/recurrent is naturally lifted from Markov chains to
MDPs, where they depend on the chosen strategy.

We define the $\transient$ objective as the set of runs that do not visit
any state infinitely often. We call an MDP \emph{universally transient} iff
it almost-surely satisfies $\transient$ under every strategy.
Thus every acyclic MDP is universally transient, but not vice-versa;
cf.~\cref{fig:gambler-ruin}. In particular, universal transience does not
just depend on the structure of the transition graph, but also on the
transition probabilities.
Universally transient MDPs have interesting properties. Many objectives
(e.g., Safety, B\"uchi, co-B\"uchi) have a lower strategy complexity than in
general MDPs; see below.

We also study the strategy complexity of the $\transient$ objective itself,
and how it interacts with other objectives, e.g., how to attain a B\"uchi objective
in a transient way.

\smallskip
{\bf\noindent Our contributions.}
\begin{enumerate}
\item
  We show that there exist uniformly $\eps$-optimal MD strategies (memoryless
  deterministic) for $\transient$, even in infinitely branching MDPs.
  This is unusual, since (apart from reachability objectives) most other
  objectives require infinite memory if the MDP is infinitely branching,
  e.g., all objectives generalizing Safety \cite{KMSW2017}.
 
  Our result is shown in several steps.
  First we show that there exist $\eps$-optimal deterministic 1-bit strategies
  for $\transient$.
  Then we show how to dispense with the 1-bit memory and obtain
  $\eps$-optimal MD strategies for $\transient$.
  Finally, we make these MD strategies uniform, i.e., independent of the
  start state.
\item
  We show that optimal strategies for $\transient$ need not exist,
  even if the MDP is finitely branching.
  If they do exist then there are also MD optimal strategies.
  More generally, there exists a single MD strategy that is optimal from
  every state that allows optimal strategies for $\transient$.
\item
If an MDP is universally transient (i.e., almost surely transient under
all strategies) then many other objectives have a lower strategy complexity
than in general MDPs, e.g., $\eps$-optimal strategies for Safety and co-B\"uchi and
optimal strategies for $\cParity{\{0,1,2\}}$ (where they exist) can be chosen
MD, even if the MDP is infinitely branching.
\end{enumerate}
For our proofs we develop some technical results that are of
independent interest.
We generalize Ornstein's plastering construction~\cite{Ornstein:AMS1969} from
reachability to tail objectives and thus obtain a general tool to
infer uniformly $\eps$-optimal MD strategies from non-uniform ones
(cf.~\cref{thm:Ornstein-plastering}).
Secondly, in \cref{sec:conditioned} we develop the notion of the
\emph{conditioned MDP} (cf.~\cite{KMSW2017}). For tail objectives, this allows to
obtain uniformly $\eps$-optimal MD strategies wrt.\ \emph{multiplicative errors}
from those with merely additive errors.

\section{Preliminaries}\label{sec:prelim}

A \emph{probability distribution} over a countable set $S$ is a function
$f:\states\to[0,1]$ with $\sum_{\state\in \states}f(\state)=1$.
We write 
$\dist(\states)$ for the set of all probability distributions over $\states$.

\smallskip
\noindent{\bf Markov Decision Processes.} We define Markov decision processes (MDPs for short) over  
countably infinite
state spaces as tuples $\mdp=\mdptuple$ where 
$S$ is the countable set of states  
 partitioned into a set~$\zstates$ of \emph{controlled states} 
and  a set~$\rstates$ of \emph{random states}.
The \emph{transition relation} is~$\transition\subseteq\states\x\states$,
and $\probp:\rstates \to \dist(\states)$ is  a  \emph{probability function}. 
We  write $\state\transition{}\state'$ if $\tuple{\state,\state'}\in \transition$,
and  refer to~$s'$ as a \emph{successor} of~$s$. 
We assume that every state has at least one successor.  
The probability function~$P$  assigns to each random state~$\state\in \rstates$
a probability distribution~$P(\state)$ over its set of successors.
A \emph{sink} is a subset $T \subseteq \states$ closed under the $\transition$ relation.

An MDP is \emph{acyclic} if the underlying  graph~$(S,\transition)$ is acyclic.
It is  \emph{finitely branching} 
if every state has finitely many successors
and \emph{infinitely branching} otherwise.
An MDP without controlled states
($\zstates=\emptyset$) is
a \emph{Markov chain}.

\smallskip
\noindent{\bf Strategies and Probability Measures.}
A \emph{run}~$\play$ is an  infinite sequence $\state_0\state_1\cdots$ of states
such that $\state_i\transition{}\state_{i+1}$ for all~$i\in \mathbb{N}$;
%
%
a \emph{partial run} is a finite prefix of a run.
We write $\play(i)=s_i$ and say that (partial) run $\state_0\state_1\cdots$ \emph{visits} $\state$ if
$\state=s_i$ for some $i$. It \emph{starts in} $s$ if $\state=s_0$. 

A \emph{strategy} 
is a function $\zstrat:\states^*\zstates \to \dist(S)$ that 
assigns to partial runs $\partialplay\state \in \states^*\zstates$ 
a distribution over the successors of $\state$. 
We write $\zallstrats{\mdp}$ for the set of all strategies in $\mdp$.
 A strategy~$\zstrat$ and an initial state $\state_0\in \states$
induce a standard probability measure on sets of infinite runs. We write $\probm_{\mdp,\state_0,\zstrat}({\playset})$ for the probability of a 
measurable set $\playset \subseteq \state_0 \states^\omega$ of runs starting from~$\state_0$.
It is   defined for the cylinders~$s_0 s_1 \ldots s_n \states^\omega\in \states^{\omega}$ as
 $\probm_{\mdp,\state_0,\zstrat}(s_0 s_1 \ldots s_n \states^\omega) \eqdef \prod_{i=0}^{n-1} \bar{\zstrat}(s_0 s_1 \ldots s_i)(s_{i+1})$, where $\bar{\zstrat}$ is the map that extends~$\zstrat$ by $\bar{\zstrat}(w s) = \probp(s)$ for all $w s \in \states^* \rstates$.
By Carath\'eodory's theorem~\cite{billingsley-1995-probability}, 
the measure for cylinders extends uniquely to a probability measure~$\probm_{\mdp,\state_0,\zstrat}$ on 
 all measurable subsets of $s_0S^{\omega}$.
We will write $\expectval_{\mdp,\state_0,\zstrat}$
for the expectation w.r.t.~$\probm_{\mdp,\state_0,\zstrat}$.

\smallskip
\noindent{\bf Strategy Classes.}
Strategies  $\zstrat:\states^*\zstates \to \dist(S)$ are in general  \emph{randomized} (R) in the sense that they take values in $\dist(\states)$. 
A strategy~$\zstrat$ is \emph{deterministic} (D) if $\zstrat(\rho)$ is a Dirac distribution 
for all partial runs~$\rho\in \states^{*} \zstates$.

We formalize the amount of \emph{memory} needed to implement strategies in Appendix~\ref{app-st-classes}.
The two classes of \emph{memoryless} and \emph{1-bit}  strategies are central to this paper.
A strategy $\zstrat$ is \emph{memoryless}~(M) 
if $\zstrat$ bases  its decision only on the last state of the run: $\zstrat(\rho s)=\zstrat(\rho' s)$
for all $\rho,\rho'\in \states^*$. 
We may view
M-strategies as functions $\zstrat: \zstates \to \dist(\states)$.
A {1-bit} strategy  $\zstrat$ may base  its decision
 also on a memory mode~$\memconf \in\{0,1\}$. Formally, a {1-bit} strategy $\zstrat$ is given as a tuple $(u,\memconf_0)$
where $\memconf_0\in \{0,1\}$ is the initial memory mode and 
$u:\{0,1\}\x \states \to \dist(\{0,1\} \x \states)$ is an update function such that
\begin{itemize}
	\item for all controlled states~$\state\in \zstates$, 
            the distribution
            $\updatefun((\memconf,\state))$ is over 
	$\{0,1\} \times \{\state'\mid \state \transition{} \state'\}$.
	\item for all random states $\state \in \rstates$, we have that
            $\sum_{\memconf'\in \{0,1\}} \updatefun((\memconf,\state))(\memconf',\state')=P(\state)(\state')$.
\end{itemize}
Note that this definition allows for updating the memory mode upon visiting
random states. 
We write $\zstrat[\memconf_0]$ for the strategy obtained from~$\zstrat$
by setting the initial memory mode to~$\memconf_0$.

\emph{MD strategies} are both memoryless and deterministic; and \emph{deterministic 1-bit} strategies   are both deterministic and 1-bit.  

\smallskip
\noindent {\bf Objectives.} 
The objective of the controller is determined by a predicate on infinite runs.
We assume familiarity with the syntax and semantics of the temporal
logic LTL \cite{CGP:book}.
Formulas are interpreted on the underlying structure $(\states,\transition)$ of the MDP~$\mdp$.
We use 
$\denotationof{\formula}{\mdp,\state} \subseteq \state \states^\omega$ to denote the set of runs starting from
$\state$ that satisfy the LTL formula $\formula$,
which is a measurable set \cite{Vardi:probabilistic}.
We also write $\denotationof{\formula}{\mdp}$ for $\bigcup_{s \in S} \denotationof{\formula}{\mdp,\state}$.
Where it does not cause confusion we will
identify $\varphi$ and $\denotationof{\formula}{}$
and just write
$\probm_{\mdp,\state,\zstrat}(\formula)$
instead of 
$\probm_{\mdp,\state,\zstrat}(\denotationof{\formula}{\mdp,\state})$.

Given a  set $\reachset \subseteq \states$   of  states, 
  the \emph{reachability} objective $\reach{\reachset}\eqdef\eventually \reachset$ is the set of  runs that visit  $\reachset$ at least once.
  The \emph{safety} objective~$\safety{\reachset} \eqdef \always \neg T$ is the set of  runs that never visit~$T$.

Let $\cset \subseteq \nat$ be a finite set of colors.
A \emph{color function} $\coloring:\states\to \cset$ assigns to each state~$\state$ its color~$\colorof\state$. The parity objective, written as
$\Parity{\coloring}$, is the set of infinite runs such that the largest
color that occurs infinitely often along the run is even.
To define this formally, let
$\even(\cset)=\{i\in \cset\mid i\equiv 0\mod{2}\}$.
For 
$\mathord{\constraint}\in\{\mathord{<},\mathord{\le},\mathord{=},\mathord{\geq},\mathord{>}\}$,
$n\in\nat$, 
and $\stateset\subseteq\states$, let
$\colorset \stateset\constraint n \eqdef \setcomp{\state\in \stateset}{\colorof\state\constraint n}$
be the set of states in $\stateset$ with color $\constraint n$.
Then 
\[
 \Parity{\coloring} \eqdef
     \bigvee_{i\in \even(\cset)}\left(\always\eventually \colorset{\states}{= i}{} \wedge
 \eventually\always \colorset{\states}{\leq  i}{}\right).
\]

We write $\cParity{\cset}$ for the parity objectives with the set of colors~$\cset \subseteq \nat$.
The classical B\"uchi and co\nobreakdash-B\"uchi  objectives correspond to
$\cParity{\{1,2\}}$ and $\cParity{\{0,1\}}$, respectively.

{\renewcommand{\denotationof}[1]{#1}
    An objective $\formula$ is called a \emph{tail objective} (in $\mdp$)
 iff for every run
$\rho'\rho$ with some finite prefix $\rho'$ we have
$\rho'\rho \in \denotationof{\formula}{} \Leftrightarrow \rho \in \denotationof{\formula}{}$.   
}
For every coloring $\coloring$, $\Parity{\coloring}$ is tail.
Reachability objectives 
are not always tail but in MDPs where the target set $\reachset$ is a sink
$\reach{\reachset}$  is tail.

\smallskip
\noindent{\bf Optimal and $\eps$-optimal Strategies.}
Given an objective~$\formula$, the \emph{value} of state~$s$ in an MDP~$\mdp$, denoted by 
$\valueof{\mdp,\formula}{s}$, is the supremum probability of achieving~$\formula$.
 Formally, we have  $\valueof{\mdp,\formula}{s} \eqdef\sup_{\sigma \in \Sigma} \probm_{\mdp,\state,\zstrat}(\formula)$ where $\Sigma$ is the set of all strategies.
For $\eps\ge 0$ and state~$s\in\states$, we say that a strategy is \emph{$\eps$-optimal} from $s$
iff $\probm_{\mdp,\state,\zstrat}(\formula) \geq \valueof{\mdp,\formula}{s} -\eps$.
A $0$-optimal strategy is called \emph{optimal}. 
An optimal strategy is \emph{almost-surely winning} iff $\valueof{\mdp,\formula}{s} = 1$. 
 
Considering an MD strategy as a function $\zstrat: \zstates \to \states$ and $\eps\ge 0$, $\zstrat$ is \emph{uniformly} $\eps$-optimal  (resp.~uniformly optimal) if it is $\eps$-optimal (resp.~optimal) from every $s\in S$.

\smallskip
Throughout the paper, we may drop the subscripts and superscripts from  notations, if it is understood from the context. The missing proofs can be found in the appendix.

\section{Transience and Universally Transient MDPs}\label{sec:transientPre}

In this section we define the transience property for MDPs, a natural generalization  of the well-understood concept of transient Markov chains. We enumerate crucial characteristics of this objective and define the notion of universally transient MDPs.

Fix a countable MDP $\mdp=\mdptuple$. Define the transience objective, denoted by $\transient$,
to be the set of runs that do not visit any state of~$\mdp$ infinitely
  often, i.e., 
  \[\transient \eqdef \bigwedge_{\state \in \states} \eventually\always \; \neg\state.\]
The~$\transient$ objective is tail, as it is closed under removing  finite prefixes of runs. Also  note that $\transient$ cannot be encoded in a parity objective.

We call $\mdp$ \emph{universally transient} iff for all states~$s_0$, for all strategies $\zstrat$, the   
  $\transient$ property holds almost-surely from~$s_0$, i.e.,
 \[
 \forall \state_0\in\states~~\forall \zstrat\in\zstratset~~\probm_{\mdp,\state_0,\zstrat}(\transient)=1.
 \]

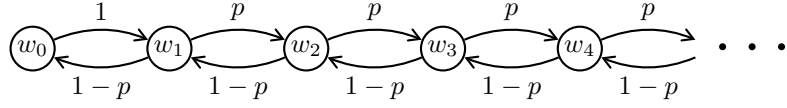
\begin{figure}[t]
\begin{center}
\begin{tikzpicture}[xscale=1.8,yscale=1.3]
\node[MDPrand] (0) at (0,0) {$w_0$};
\node[MDPrand] (1) at (1,0) {$w_1$};
\node[MDPrand] (2) at (2,0) {$w_2$};
\node[MDPrand] (3) at (3,0) {$w_3$};
\node[MDPrand] (4) at (4,0) {$w_4$};
\node[MDPrand,draw=none] (5) at (5,0) {};
\node (6) at (5.3,0) {\Huge $\cdots$};
\draw[->] (0) edge[bend left] node[above] {$1$} (1);
\draw[->] (1) edge[bend left] node[above] {$p$} (2);
\draw[->] (2) edge[bend left] node[above] {$p$} (3);
\draw[->] (3) edge[bend left] node[above] {$p$} (4);
\draw[->] (4) edge[bend left] node[above] {$p$} (5);
\draw[->] (5) edge[bend left] node[below] {$1-p$} (4);
\draw[->] (4) edge[bend left] node[below] {$1-p$} (3);
\draw[->] (3) edge[bend left] node[below] {$1-p$} (2);
\draw[->] (2) edge[bend left] node[below] {$1-p$} (1);
\draw[->] (1) edge[bend left] node[below] {$1-p$} (0);
\end{tikzpicture}
\end{center}
\vspace{-.4cm}
\caption{Gambler's Ruin with restart: The state~$w_i$ illustrates that the controller's wealth is~$i$, and 
the coin tosses are in  the controller's favor with probability~$p$. For all $i$,  $\probm_{w_i}(\transient)=0$ if $p\leq \frac{1}{2}$; and $\probm_{w_i}(\transient)=1$ otherwise.}
\label{fig:gambler-ruin}
\vspace{-.2cm}
\end{figure}

The MDP in Figure~\ref{fig:gambler-ruin} models the classical Gambler's Ruin
Problem with restart; see~\cite[Chapter~14]{Feller:book}.
It is well-known that 
if the controller starts with wealth~$i$
 and if $p\leq \frac{1}{2}$, the probability  of ruin (visiting the state~$w_0$)  
is $\probm_{w_i}(\eventually \, w_0)=1$.
Consequently, the probability of re-visiting~$w_0$ infinitely often is~$1$,  
implying that  $\probm_{w_i}(\transient)=0$.
In contrast, for the case with $p> \frac{1}{2}$, for all states~$w_i$,
the probability of re-visiting~$w_i$ is strictly below~$1$. 
Hence, the $\transient$ property holds almost-surely. 
This example indicates that the transience property depends on the probability
values of the transitions and not just on the underlying transition graph,
and thus may require arithmetic reasoning. 
In particular, the MDP in Figure~\ref{fig:gambler-ruin} 
is universally transient iff $p>\frac{1}{2}$.

In general, optimal strategies for $\transient$ need not exist:
\begin{lemma}\label{lem:no-optimal}
There exists a finitely branching countable MDP with  initial state~$\state_0$ such that
 \begin{itemize}
 \item $\valueof{\transient}{\state}=1$ for all controlled states~$\state$, 
 \item there does not exist any optimal strategy~$\sigma$ such that $\probm_{\state_0,\sigma}(\transient)=1$.
 \end{itemize}
\end{lemma}
\begin{proof}
	Consider a countable MDP~$\mdp$ with set~$S=\{\ell_i,\ell'_i,r_i,x_i \mid i\geq 1\} \cup\{\ell_0,\bot\}$ of states; see Figure~\ref{fig:no-optimal}.
For all $i\geq 1$ the state $x_{i+1}$ is the unique successor of~$x_i$ so that $(x_i)_{i\geq 1}$ form an acyclic ladder;  the value of $\transient$ is $1$ for all $x_i$. 
The state $\bot$ is  sink, and  its value is~$0$. 
	The states~$(r_i)_{i\geq 1}$ are all random, and  $r_i \step{1-2^{-i}} x_i$ and $r_i \step{2^{-i}} \bot$. Observe that  
	 the value of $\transient$ is~$1-2^{-i}$ for the $r_i$.	
	
	The states $(\ell_i)_{i\in \mathbb{N}}$ are controlled whereas the states $(\ell'_i)_{i\geq 1}$ are random. 
	By interleaving of these states, we construct a ``recurrent ladder'' of decisions: $\ell_0 \to \ell_1$ and for all $i\geq 1$,  state~$\ell_i$ has two successors~$\ell'_{i}$ and $r_i$. 
	 In random states~$\ell'_i$, as in Gambler's Ruin with a fair coin, the successors are $\ell_{i-1}$ or $\ell_{i+1}$, each with equal probability. 
	In each state~$(\ell_i)_{i\geq 1}$, the controller decides to either stay on the ladder   by going to~$\ell'_{i}$ or leaves the ladder to~$r_i$. As in Figure~\ref{fig:gambler-ruin}, if the controller stays on the ladder forever, the probability of $\transient$ is~$0$.

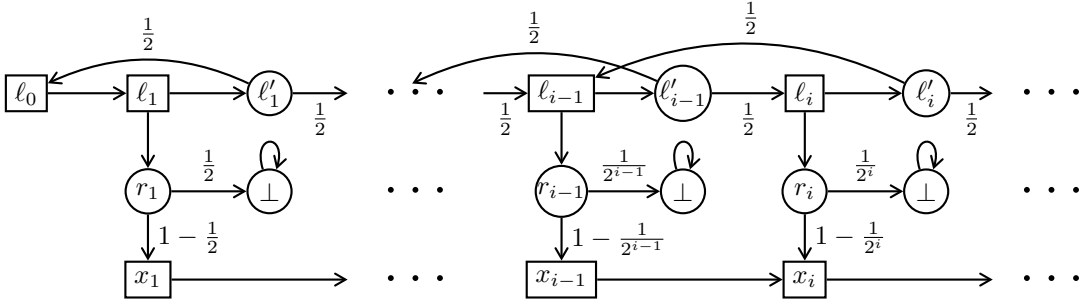
\begin{figure}
 \centering
 \begin{tikzpicture}[xscale=1.6,yscale=1.3]
 \node[MDPcont] (s0) at (0,2.6) {$~\ell_0~$};
\node[MDPcont] (s1) at (1,2.6) {$~\ell_{1}~$};
\node[MDPrand] (s2) at (2,2.6) {$\ell'_{1}$};
\node[MDPrand, draw=none] (s3) at (3.2,2.6) {\huge $~~\cdots~~$};
\node[MDPrand, draw=none] (ss3) at (3,2.6) {};
\node[MDPcont] (s4) at (4.4,2.6) {$~\ell_{i-1}~$};
\node[MDPrand] (s5) at (5.4,2.6) {$\ell'_{i-1}$};
\node[MDPcont] (s6) at (6.4,2.6) {$~\ell_i~$};
\node[MDPrand] (s7) at (7.4,2.6) {$~\ell'_{i}~$};
\node (s8) at (8.4,2.6) {\huge $~\cdots$};

\node[MDPrand] (r1) at (1,1.6) {$~r_{1}~$};
\node[MDPrand] (bot1) at (2,1.6) {$\bot$};
\node[MDPrand, draw=none] (r0) at (3.2,1.6) {\huge $~~\cdots~~$};
\node[MDPrand] (r3) at (4.4,1.6) {$r_{i-1}$};
\node[MDPrand] (bot3) at (5.4,1.6) {$\bot$};
\node[MDPrand] (r5) at (6.4,1.6) {$~r_i~$};
\node[MDPrand] (bot5) at (7.4,1.6) {$\bot$};
\node (r6) at (8.4,1.6) {\huge $~\cdots$};

\node[MDPcont] (x1) at (1,0.7) {$~x_{1}~$};
\node[MDPrand, draw=none] (x0) at (3.2,0.7) {\huge $~~\cdots~~$};
\node[MDPcont] (x3) at (4.4,0.7) {$~x_{i-1}~$};
\node[MDPcont] (x5) at (6.4,0.7) {$~x_i~$};
\node (x6) at (8.4,0.7) {\huge $~\cdots$};

\draw[->] (s0) edge  (s1);
\draw[->] (s1) edge  (s2);
\draw[->] (s2) edge node[below] {$\frac{1}{2}$}(s3);
\draw[->] (s3) edge node[below] {$\frac{1}{2}$} (s4);
\draw[->] (s4) edge  (s5);
\draw[->] (s5) edge node[below] {$\frac{1}{2}$}(s6);
\draw[->] (s6) edge  (s7);
\draw[->] (s7) edge node[below] {$\frac{1}{2}$}(s8);

\draw[->] (s2) edge[bend right] node[above] {$\frac{1}{2}$} (s0);
\draw[->] (s5) edge[bend right] node[above] {$\frac{1}{2}$} (ss3);
\draw[->] (s7) edge[bend right] node[above] {$\frac{1}{2}$} (s4);

\draw[->] (s1) edge  (r1);
\draw[->] (s4) edge  (r3);
\draw[->] (s6) edge  (r5);

\draw[->] (x1) edge  (x0);
\draw[->] (x3) edge  (x5);
\draw[->] (x5) edge  (x6);
\draw[->] (r1) edge node[right] {{$1-\frac{1}{2}$}} (x1);
\draw[->] (r1) edge node[above] {$\frac{1}{2}$} (bot1);
\draw[->] (r3) edge node[right] {$1-\frac{1}{2^{i-1}}$} (x3);
\draw[->] (r3) edge node[above] {$\frac{1}{2^{i-1}}$} (bot3);
\draw[->] (r5) edge node[right] {$1-\frac{1}{2^{i}}$} (x5);
\draw[->] (r5) edge node[above] {$\frac{1}{2^{i}}$} (bot5);

\draw[->] (bot1) edge[loop above, looseness=10]  (bot1);
\draw[->] (bot3) edge[loop above, looseness=10]  (bot3);
\draw[->] (bot5) edge[loop above, looseness=10]  (bot5);

\end{tikzpicture}
\vspace{-1cm}
\caption{A partial illustration of the MDP in Lemma~\ref{lem:no-optimal}, in which there is no optimal strategy for $\transient$, starting from states~$\ell_i$. 
For readability, we have three copies of the state~$\bot$. We call the ladder consisting of the interleaved controlled states~$\ell_i$ and random states~$\ell'_i$ a ``recurrent ladder'': if the controller stays on this ladder forever, it faithfully simulates a Gambler's Ruin with a fair coin, and the probability of $\transient$ will be~$0$. }
\label{fig:no-optimal}
\vspace{-.2cm}
 \end{figure}
	
Starting in~$\ell_0$, for all $i>0$, strategy~$\sigma_i$ that stays on the
ladder until visiting $\ell_i$ (which happens eventually almost surely)
and then leaves the ladder to~$r_i$ achieves $\transient$  with probability~$1-2^{i}$. Hence,  $\valueof{\transient}{\ell_0}=1$.

Recall that transience cannot be achieved with a positive probability by staying on the acyclic ladder forever.
But any strategy that leaves the ladder with a positive probability comes with
a positive probability of falling into~$\bot$, thus is not optimal either.
Thus there is no optimal strategy for $\transient$.
 \end{proof}

\noindent{\bf Reduction to Finitely Branching MDPs.}
In our main results, we will prove that for the $\transient$ property 
 there always exist $\varepsilon$-optimal MD strategies in finitely branching countable MDPs;
and if an optimal strategy exists, there will exist an optimal MD strategy. 
We generalize these results to infinitely branching countable MDPs
by the following reduction: 


\begin{restatable}{lemma}{lemreductionfinitebranch}\label{lem:reduction-finite-branch}
Given an infinitely branching countable MDP~$\mdp$ with an initial state~$s_0$, 
there exists a finitely branching countable~$\mdp'$ with a set~$S'$ of states such that $s_0\in S'$ 
and   
\begin{enumerate}
	\item each strategy~$\alpha_1$ in~$\mdp$ is mapped to a unique strategy~$\beta_1$ in~$\mdp'$ where
	\[\probm_{s_0,\alpha_1}(\transient)=\probm_{s_0,\beta_1}(\transient),\]
	 
	\item  and conversely, every MD strategy~$\beta_2$ in $\mdp'$  is mapped to an MD strategy~$\alpha_2$ in~$\mdp$ where
	\[\probm_{s_0,\alpha_2}(\transient)\geq \probm_{s_0,\beta_2}(\transient).\]
\end{enumerate}
\end{restatable}
\begin{proof}[Proof sketch]
See \cref{app-transientPre} for the complete construction.
In order to construct~$\mdp'$ from~$\mdp$, 
for each controlled state~$s\in S$ in~$\mdp$ that has infinitely many successors~$(s_i)_{i\geq 1}$,
 a ``recurrent ladder'' is introduced; see Figure~\ref{fig:reduction-inf}. 
Since the probability of $\transient$ is $0$ for 
all those runs that eventually stay forever on a recurrent ladder, the
controller should exit such ladders to play optimally for~$\transient$.
Infinitely branching random states can be dealt with in an easier way.
\end{proof}

\begin{figure}[t]
 \centering
 \begin{tikzpicture}[xscale=1.6,yscale=1.3]
 
 \draw [blue, thick, dashed] (-0.3,6.6) rectangle (8.8,5.4);

 \node[MDPcont] (s) at (0,6.3) {$~s~$}; 
\node[MDPcont] (s1) at (1,5.7) {$~s_{1}~$};
\node[MDPrand, draw=none] (r0) at (2.7,5.8) {\huge $~\cdots~$};
\node[MDPrand] (s3) at (4.4,5.7) {$s_{i-1}$};
\node[MDPcont] (s5) at (6.4,5.7) {$~s_i~$};
\node (r6) at (7.7,5.8) {\huge $~\cdots$};
\node[MDPrand, draw=none]  at (8,6.2) {\Large MDP $\mdp$};
 
 \draw[-](s) edge (7,6.3);
\draw[->] (1,6.3) --  (s1);
\draw[->] (4.4,6.3) --  (s3);
\draw[->] (6.4,6.3) --  (s5);

\node[MDPrand, draw=none]  at (4.5,5.1) {\textcolor{blue}{{\huge $\Downarrow$} reduction}};
 
 \draw [blue, thick, dotted] (-0.3,4.8) rectangle (8.8,2.6);
 \node[MDPrand, draw=none]  at (8,4.4) {\Large MDP $\mdp'$};
  \node[MDPcont] (s) at (0,4.5) {$~s~$}; 
  
 \node[MDPcont] (s0) at (0,3.7) {$~\ell_0~$};
\node[MDPcont] (s1) at (1,3.7) {$~\ell_{1}~$};
\node[MDPrand] (s2) at (2,3.7) {$\ell'_{1}$};
\node[MDPrand, draw=none] (s3) at (3.2,3.7) {\huge $~~\cdots~~$};
\node[MDPrand, draw=none] (ss3) at (3,3.7) {};
\node[MDPcont] (s4) at (4.4,3.7) {$~\ell_{i-1}~$};
\node[MDPrand] (s5) at (5.4,3.7) {$\ell'_{i-1}$};
\node[MDPcont] (s6) at (6.4,3.7) {$~\ell_i~$};
\node[MDPrand] (s7) at (7.4,3.7) {$~\ell'_{i}~$};
\node (s8) at (8.4,3.7) {\huge $~\cdots$};

\node[MDPcont] (r1) at (1,2.9) {$~s_{1}~$};
\node[MDPrand, draw=none] (r0) at (3.2,2.9) {\huge $~\cdots~$};
\node[MDPrand] (r3) at (4.4,2.9) {$s_{i-1}$};
\node[MDPcont] (r5) at (6.4,2.9) {$~s_i~$};
\node (r6) at (8.4,2.9) {\huge $~\cdots$};

\draw[->] (s) edge  (s0);
\draw[->] (s0) edge  (s1);
\draw[->] (s1) edge  (s2);
\draw[->] (s2) edge node[below] {$\frac{1}{2}$}(s3);
\draw[->] (s3) edge node[below] {$\frac{1}{2}$} (s4);
\draw[->] (s4) edge  (s5);
\draw[->] (s5) edge node[below] {$\frac{1}{2}$}(s6);
\draw[->] (s6) edge  (s7);
\draw[->] (s7) edge node[below] {$\frac{1}{2}$}(s8);

\draw[->] (s2) edge[bend right] node[above] {$\frac{1}{2}$} (s0);
\draw[->] (s5) edge[bend right] node[above] {$\frac{1}{2}$} (ss3);
\draw[->] (s7) edge[bend right] node[above] {$\frac{1}{2}$} (s4);

\draw[->] (s1) edge  (r1);
\draw[->] (s4) edge  (r3);
\draw[->] (s6) edge  (r5);

\end{tikzpicture}
\vspace{-.7cm}
\caption{A partial illustration of the reduction in Lemma~\ref{lem:reduction-finite-branch}. }
\label{fig:reduction-inf}
\vspace{-.2cm}
 \end{figure}
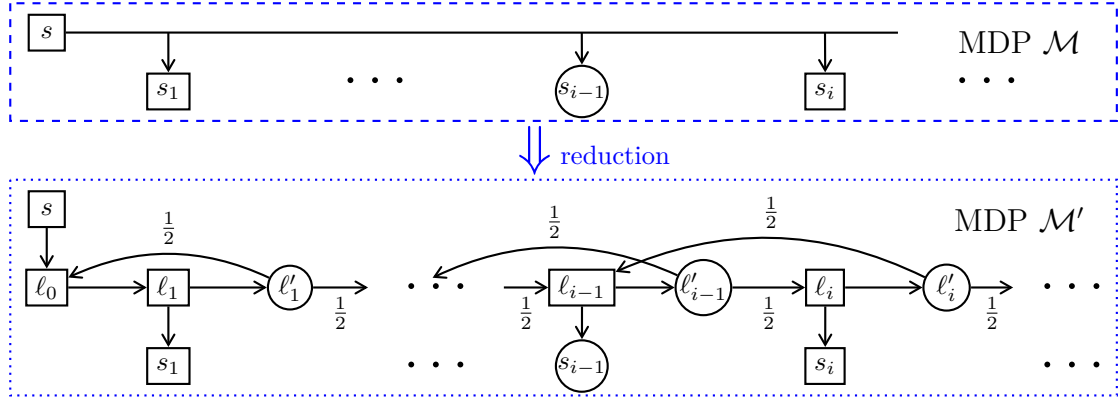

\medskip
\noindent{\bf Properties of Universally Transient MDPs.}

Notice that acyclicity implies universal transience, but not vice-versa.

\begin{restatable}{lemma}{lemstructuraltransience}\label{lem:structural-transience}
  For every countable MDP $\mdp=\mdptuple$, the following conditions are equivalent.
  \begin{enumerate}
  \item
    $\mdp$ is universally transient, i.e., $\forall \state_0, \forall\zstrat.\, \probm_{\mdp,\state_0,\zstrat}(\transient)=1$.
  \item
    For every initial state $\state_0$ and state $\state$,
    the objective of re-visiting $\state$ infinitely often has value zero, i.e.,
$\forall \state_0,\state\,\sup_{\zstrat}\probm_{\mdp,\state_0,\zstrat}(\always\eventually(\state))=0$.
  \item
    For every state $\state$ the value of the objective to re-visit $\state$
    is strictly below $1$, i.e.,\\
    ${\it Re}(s) \eqdef \sup_{\zstrat}\probm_{\mdp,\state,\zstrat}(\next\eventually(\state)) < 1$.
  \item
    For every state $\state$ there exists a finite bound $B(\state)$ such that
    for every state $\state_0$ and strategy $\zstrat$ from $\state_0$
    the expected number of visits to $\state$ is $\le B(\state)$.
  \item
    For all states $\state_0, \state$, under every strategy $\zstrat$ from $\state_0$
    the expected number of visits to $\state$ is finite.
\end{enumerate}
\end{restatable}

\begin{proof}
Towards $(1)\Rightarrow (2)$, consider an arbitrary strategy $\zstrat$
from the initial state $\state_0$ and some state $\state$.
By (1) we have
$\forall \zstrat . \probm_{\mdp,\state_0,\zstrat}(\transient)=1$
and thus
$0 = \probm_{\mdp,\state_0,\zstrat}(\neg\transient) =
\probm_{\mdp,\state_0,\zstrat}(\bigcup_{\state' \in \states} \always\eventually(\state'))
\ge
\probm_{\mdp,\state_0,\zstrat}(\always\eventually(\state))$
which implies (2).

Towards $(2)\Rightarrow (1)$, consider an arbitrary strategy $\zstrat$
from the initial state $\state_0$.
By (2) we have
$0 = \sum_{\state \in \states} \probm_{\mdp,\state_0,\zstrat}(\always\eventually(\state))
\ge
\probm_{\mdp,\state_0,\zstrat}(\bigcup_{\state \in \states}\always\eventually(\state))
=
\probm_{\mdp,\state_0,\zstrat}(\neg\transient)
$ and thus
$\probm_{\mdp,\state_0,\zstrat}(\transient)=1$.

We now show the implications $(2) \Rightarrow (3) \Rightarrow (4) \Rightarrow
(5) \Rightarrow (2)$.

Towards $\neg(3)\Rightarrow\neg(2)$,
$\neg(3)$ implies $\exists s. {\it Re}(s)=1$ and thus
$\forall\eps>0.\exists \zstrat_\eps\, \probm_{\mdp,\state,\zstrat_\eps}(\next\eventually(\state)) \ge 1-\eps$.
Let $\eps_i \eqdef 2^{-(i+1)}$. We define the strategy $\zstrat$ to play like
$\zstrat_{\eps_i}$ between the $i$-th and $(i+1)$th visit to $\state$.
Since $\sum_{i=1}^\infty \eps_i < \infty$,
we have $\prod_{i=1}^\infty (1-\eps_i) >0$. 
Therefore $\probm_{\mdp,\state,\zstrat}(\always\eventually(\state))
\ge \prod_{i=1}^\infty (1-\eps_i) >0$,
which implies $\neg(2)$, where $s_0 = s$.


Towards $(3)\Rightarrow (4)$, regardless of $\state_0$ and the chosen strategy,
the expected number of visits to $\state$ is upper-bounded by
$B(\state) \eqdef \sum_{n=0}^\infty (n+1) \cdot ({\it Re}(\state))^n < \infty$.

The implication $(4)\Rightarrow (5)$ holds trivially.

Towards $\neg(2)\Rightarrow \neg(5)$,
by $\neg(2)$ there exist states $\state_0,\state$
and a strategy $\zstrat$ such that 
$\probm_{\mdp,\state_0,\zstrat}(\always\eventually(\state)) >0$.
Thus the expected number of visits to $\state$ is infinite, which implies~$\neg(5)$.
\end{proof}

We remark that
if an MDP is \emph{not} universally transient 
(unlike in \cref{lem:structural-transience}(5)), for a strategy~$\sigma$, 
the expected number of visits to some state can be infinite,
even if $\sigma$ attains $\transient$ almost surely.
  
Consider the MDP $\mdp$ with controlled states
$\{\state_0, \state_1, \dots\}$, initial state $\state_0$ 
and transitions $\state_0 \to \state_0$ and $\state_k \to \state_{k+1}$ for
every $k \ge 0$. 
We define a strategy $\zstrat$ that, while in state $\state_0$,
proceeds in rounds $i=1,2,\dots$.
In the $i$-th round it tosses a fair coin.
If Heads then it goes to $\state_1$.
If Tails then it loops around $\state_0$ exactly $2^i$ times and
then goes to round $i+1$.
In every round the probability of going to $\state_1$ is $1/2$
and therefore the probability of staying in $\state_0$ forever is $(1/2)^\infty = 0$.
Thus $\probm_{\mdp,\state_0,\zstrat}(\transient)=1$.
However, the expected number of visits to $\state_0$
is $\ge \sum_{i=1}^\infty \left(\frac{1}{2}\right)^i \cdot 2^i = \infty$.

\section{MD Strategies for Transience}\label{transientMD}
We show that there exist uniformly $\eps$-optimal MD strategies for $\transient$
and that optimal strategies, where they exist, can also be chosen MD.

First we show that there exist $\eps$-optimal
deterministic 1-bit strategies for $\transient$
(in \cref{cor:transient-buchi})
and then we show how to dispense with the 1-bit memory
(in \cref{lem:TransientMD_nonuniform}).

It was shown in \cite{KMST:ICALP2019} that there exist $\eps$-optimal
deterministic 1-bit strategies for B\"uchi objectives in 
\emph{acyclic} countable MDPs (though not in general MDPs).
These 1-bit strategies will be similar to the 1-bit strategies for 
$\transient$ that we aim for in (not necessarily acyclic) countable 
MDPs. In \cref{thm:MDP-one-bit-Buchi} below we first strengthen the 
result from \cite{KMST:ICALP2019} and construct $\eps$-optimal 
deterministic 1-bit strategies for objectives
$\reset(F) \cap \transient$.
From this we obtain deterministic
1-bit strategies for $\transient$ (\cref{cor:transient-buchi}). 

\begin{restatable}{lemma}{thmMDPonebitBuchi}\label{thm:MDP-one-bit-Buchi}
Let $\mdp$ be a countable MDP, $\initstates$ a finite set of initial states,
$F$ a set of states and $\eps>0$.
Then there exists a deterministic 1-bit
strategy for $\reset(F) \cap \transient$ that is $\eps$-optimal from every $\state \in \initstates$.
\end{restatable}
\begin{proof}[Proof sketch]
The full proof can be found in \cref{app:buchi-transient}.
It follows the proof of \cite[Theorem 5]{KMST:ICALP2019}, 
which considers $\reset(F)$ conditions for \emph{acyclic} (and hence universally transient) MDPs.
The only part of that proof that
requires modification is \cite[Lemma 10]{KMST:ICALP2019},
which is replaced here by \cref{lem-bubble-extension}
to deal with general MDPs.

In short, from every $\state \in \initstates$ there exists an
$\eps$-optimal strategy $\zstrat_\state$ for $\formula \eqdef \reset(F) \cap \transient$.
We observe the behavior of the
finitely many $\zstrat_\state$ for $\state \in \initstates$
on an infinite, increasing sequence of finite subsets of $\states$.
Based on \cref{lem-bubble-extension}, we can define a second stronger
objective $\formula' \subseteq \formula$ and show
$\forall_{\state \in \initstates}\,\probm_{\mdp,\state,\zstrat_\state}(\formula') \ge \valueof{\mdp,\formula}{\state}-2\eps$.
We then construct a deterministic 1-bit
strategy $\zstrat'$ that is optimal for $\formula'$ from all
$\state \in \initstates$ and thus $2\eps$-optimal for $\formula$.
Since $\eps$ can be chosen arbitrarily small, the result follows.
\end{proof}

Unlike for the $\transient$ objective alone (see below),
the 1-bit memory is strictly necessary for the
$\reset(F) \cap \transient$ objective in \cref{thm:MDP-one-bit-Buchi}.
The 1-bit lower bound for $\reset(F)$ objectives in \cite{KMST:ICALP2019}
holds even for acyclic MDPs where $\transient$ is trivially true.

\begin{corollary}\label{cor:transient-buchi}
Let $\mdp$ be a countable MDP, $\initstates$ a finite set of initial states,
$F$ a set of states and $\eps>0$.
\begin{enumerate}
\item
If
$\forall \state \in \initstates\ \valueof{\mdp,\reset(F)}{\state} = \valueof{\mdp,\reset(F) \cap \transient}{\state}$
then there exists a deterministic 1-bit
strategy for $\reset(F)$ that is $\eps$-optimal from every
$\state \in \initstates$.
\item
If $\mdp$ is universally transient then there exists a deterministic 1-bit
strategy for $\reset(F)$ that is $\eps$-optimal from every
$\state \in \initstates$.
\item
There exists a deterministic 1-bit
strategy for $\transient$ that is $\eps$-optimal from every
$\state \in \initstates$.
\end{enumerate}
\end{corollary}
\begin{proof}
Towards (1), since $\forall \state \in \initstates\ \valueof{\mdp,\reset(F)}{\state} =
\valueof{\mdp,\reset(F) \cap \transient}{\state}$,
strategies that are $\eps$-optimal for $\reset(F) \cap \transient$ are also
$\eps$-optimal for $\reset(F)$.
Thus the result follows from \cref{thm:MDP-one-bit-Buchi}.

Item (2) follows directly from (1), since the precondition always holds in
universally transient MDPs.

Towards (3), let $F \eqdef \states$.
Then we have $\reset(F) \cap \transient = \transient$
and we obtain from \cref{thm:MDP-one-bit-Buchi}
that there exists a deterministic 1-bit
strategy for $\transient$ that is 
$\eps$-optimal from every $\state \in I$.
\end{proof}

Note that every acyclic MDP is universally transient and thus
\cref{cor:transient-buchi}(2)
implies the upper bound
on the strategy complexity of $\reset(F)$ from \cite{KMST:ICALP2019} (but not
vice-versa).

In the next step we show how to dispense with the 1-bit memory and obtain
non-uniform $\eps$-optimal MD strategies for $\transient$.

\begin{lemma}\label{lem:TransientMD_nonuniform}
Let $\mdp=\mdptuple$ be a countable MDP with initial state $\state_0$, and $\eps >0$.
There exists an MD strategy $\zstrat$ that is $\eps$-optimal for $\transient$
from $\state_0$, i.e.,
$\probm_{\mdp,\state_0,\zstrat}(\transient)
\ge
\valueof{\mdp,\transient}{\state_0} - \eps
$.
\end{lemma}
\begin{proof}
By \cref{lem:reduction-finite-branch} it suffices to prove
the property for finitely branching MDPs. Thus without restriction
in the rest of the proof
we assume that $\mdp$ is finitely branching.

Let $\eps' \eqdef \eps/2$.
We instantiate \cref{cor:transient-buchi}(3) with $I \eqdef \{\state_0\}$
and obtain that there exists an $\eps'$-optimal deterministic 1-bit
strategy $\hat{\zstrat}$ for $\transient$ from $\state_0$.

We now construct a slightly modified MDP $\mdp'$ as follows.
Let $\states_{\it bad} \subseteq \states$ be the subset of states where
$\hat{\zstrat}$ attains zero for $\transient$ in \emph{both} memory modes,
i.e.,
$\states_{\it bad} \eqdef \{\state \in \states \mid
\probm_{\mdp,\state,\zstrat[0]}(\transient)=\probm_{\mdp,\state,\zstrat[1]}(\transient)=0\}$.
Let $\states_{\it good} \eqdef \states\setminus\states_{\it bad}$.
We obtain $\mdp'$ from $\mdp$ by making all states in $\states_{\it bad}$
losing sinks (for $\transient$), by deleting all outgoing edges and adding a
self-loop instead.
It follows that
\begin{equation}\label{eq:hat-no-change}
\probm_{\mdp,\state_0,\hat{\zstrat}}(\transient)
= \probm_{\mdp',\state_0,\hat{\zstrat}}(\transient)
\end{equation}
\vspace{-9mm}
\begin{equation}\label{eq:M-better-Mprime}
\forall \zstrat.\ \probm_{\mdp,\state_0,\zstrat}(\transient)
\ge \probm_{\mdp',\state_0,\zstrat}(\transient)
\end{equation}

In the following we show that it is possible to play in such a way that,
for every $\state \in \states_{\it good}$, the expected number of visits to
$\state$ is \emph{finite}.
We obtain the deterministic 1-bit strategy $\zstrat'$ in $\mdp'$
by modifying $\hat{\zstrat}$ as follows.
In every state $\state$ and memory mode $x \in \{0,1\}$ where
$\hat{\zstrat}[x]$ attains $0$ for $\transient$
and $\hat{\zstrat}[1-x]$ attains $>0$
the strategy $\zstrat'$ sets the memory bit to $1-x$.
(Note that only states $\state \in \states_{\it good}$ can be affected by this change.)
It follows that
\begin{equation}\label{eq:prime-better-hat}
\forall \state \in \states.\ \probm_{\mdp',\state,\zstrat'}(\transient)
\ge \probm_{\mdp',\state,\hat{\zstrat}}(\transient)
\end{equation}

Moreover, from all states in $\states_{\it good}$ in $\mdp'$ the strategy $\zstrat'$
attains a strictly positive probability of $\transient$ in \emph{both} memory
modes,
i.e., for all $\state \in \states_{\it good}$ we have 
\[
t(\state,\zstrat') \eqdef \min_{x \in \{0,1\}}\probm_{\mdp',\state,\zstrat'[i]}(\transient)
>0.
\]
Let $r(\state,\zstrat',x)$ be the probability, when playing $\zstrat'[x]$ from
state $\state$, of reaching $\state$ again in the \emph{same} memory mode $x$.
For every $\state \in \states_{\it good}$ we have
$r(\state,\zstrat',x) < 1$, since $t(\state,\zstrat') >0$.

Let $R(\state)$ be the expected number of visits to state $\state$
when playing $\zstrat'$ from $\state_0$ in $\mdp'$,
and $R_x(\state)$ the expected number of visits to $\state$
in memory mode $x \in \{0,1\}$.
For all $\state \in \states_{\it good}$ we have that
\begin{equation}\label{eq:Rfinite}
R(\state)
=
R_0(\state) + R_1(\state)
\le
\sum_{n=1}^\infty n \cdot r(\state,\zstrat',0)^{n-1}
+
\sum_{n=1}^\infty n \cdot r(\state,\zstrat',1)^{n-1}
<
\infty
\end{equation}
where the first equality holds by linearity of expectations.
Thus the expected number of visits to $\state$ is \emph{finite}.

Now we upper-bound the probability of visiting $\states_{\it bad}$.
We have
$
\probm_{\mdp',\state_0,\zstrat'}(\transient)
\ge
\probm_{\mdp',\state_0,\hat{\zstrat}}(\transient)
=
\probm_{\mdp,\state_0,\hat{\zstrat}}(\transient)
\ge
\valueof{\mdp,\transient}{\state_0} - \eps'
$
by \eqref{eq:prime-better-hat}, \eqref{eq:hat-no-change} and the $\eps'$-optimality of $\hat{\zstrat}$.
Since states in $\states_{\it bad}$ are losing sinks in $\mdp'$,
it follows that
\begin{equation}\label{eq:bound-on-sbad}
\probm_{\mdp',\state_0,\zstrat'}(\eventually \states_{\it bad}) \le
1-\probm_{\mdp',\state_0,\zstrat'}(\transient)
\le
1-\valueof{\mdp,\transient}{\state_0} + \eps'
\end{equation}

We now augment the MDP $\mdp'$ by assigning costs to transitions as follows.
Let $i: \states \to \N$ be an enumeration of the state space, i.e., a
bijection.
Let $\states_{\it good}' \eqdef \{\state \in \states_{\it good} \mid
R(\state)>0\}$ be the subset of states in $\states_{\it good}$
that are visited with non-zero probability when playing $\zstrat'$ from
$\state_0$.
Each transition $\state' \to \state$ is assigned a cost:
\begin{itemize}
\item
If $\state' \in \states_{\it bad}$ then $\state \in \states_{\it bad}$ by
def.\ of $\mdp'$. We assign cost $0$.
\item
If $\state' \in \states_{\it good}$ and $\state \in \states_{\it bad}$
we assign cost $K/(1-\valueof{\mdp,\transient}{\state_0} + \eps')$
for $K \eqdef (1+\eps')/\eps'$.
\item
If $\state' \in \states_{\it good}$ and $\state \in \states_{\it good}'$
we assign cost $2^{-i(\state)}/R(\state)$.
This is well defined, since $R(\state)>0$.
\item
$\state' \in \states_{\it good}$ and $\state \in \states_{\it good}\setminus\states_{\it good}'$
we assign cost $1$.
\end{itemize}
Note that all transitions leading to states in $\states_{\it good}$ are
assigned a non-zero cost, since $R(\state)$ is finite by \eqref{eq:Rfinite}.

When playing $\zstrat'$ from $\state_0$ in $\mdp'$, the expected total cost
is upper-bounded by
\[
\probm_{\mdp',\state_0,\zstrat'}(\eventually \states_{\it bad}) \cdot
K/(1-\valueof{\mdp,\transient}{\state_0} + \eps')
+
\sum_{\state \in \states_{\it good}'} R(\state) \cdot 2^{-i(\state)}/R(\state)
\]
The first part is $\le K$ by \eqref{eq:bound-on-sbad}
and the second part is $\le 1$, since $R(\state) < \infty$ by
\eqref{eq:Rfinite}.
Therefore the expected total cost is $\le K + 1$, i.e., 
$\zstrat'$ witnesses that it is possible to attain a finite expected cost
that is upper-bounded by $K+1$.

Now we define our MD strategy $\zstrat$.
Let $\zstrat$ be an optimal MD strategy on $\mdp'$ (from $\state_0$) that minimizes the
expected cost. It exists, as a finite
expected cost is attainable and 
$\mdp'$ is finitely branching; see \cite[Theorem~7.3.6]{Puterman:book}. 

We now show that $\zstrat$ attains $\transient$ with high probability in $\mdp'$
(and in $\mdp$).
Since $\zstrat$ is cost-optimal, its attained cost from $\state_0$ is
upper-bounded by that of $\zstrat'$, i.e., $\le K+1$.
Since the cost of entering $\states_{\it bad}$ is
$K/(1-\valueof{\mdp,\transient}{\state_0} + \eps')$,
we have
$\probm_{\mdp',\state_0,\zstrat}(\eventually \states_{\it bad})
\cdot K/(1-\valueof{\mdp,\transient}{\state_0} + \eps')
\le K+1$
and thus
\begin{equation}\label{eq:bound-on-sbad2}
\probm_{\mdp',\state_0,\zstrat}(\eventually \states_{\it bad}) \le
\frac{K+1}{K}(1-\valueof{\mdp,\transient}{\state_0} + \eps')
\end{equation}
For every state $\state \in \states_{\it good}$, all transitions into $\state$
have the same fixed non-zero cost. Thus every run that visits some
state $\state \in \states_{\it good}$ infinitely often has infinite cost.
Since the expected cost of playing $\zstrat$ from $\state_0$ is $\le K+1$, such runs must be a
null-set, i.e.,
\begin{equation}\label{eq:good-nullset}
\probm_{\mdp',\state_0,\zstrat}(\neg\transient\,\wedge\,\always \states_{\it
good}) = 0
\end{equation}
Thus
\begin{align*}
&\probm_{\mdp,\state_0,\zstrat} (\transient)\\
& \ge \probm_{\mdp',\state_0,\zstrat}(\transient)  & \text{by \eqref{eq:M-better-Mprime}}\\
& = 1- \probm_{\mdp',\state_0,\zstrat}(\eventually \states_{\it bad})
& \text{by \eqref{eq:good-nullset}} \\
& \ge 1- \frac{K+1}{K}(1-\valueof{\mdp,\transient}{\state_0} + \eps')
& \text{by \eqref{eq:bound-on-sbad2}}\\
& = \valueof{\mdp,\transient}{\state_0} - \eps' - (1/K)(1-\valueof{\mdp,\transient}{\state_0} + \eps')\\
& \ge \valueof{\mdp,\transient}{\state_0} - \eps' - (1/K)(1+\eps')\\
& = \valueof{\mdp,\transient}{\state_0} - 2\eps' & \text{def.\ of $K$}\\
& = \valueof{\mdp,\transient}{\state_0} - \eps   & \text{def.\ of $\eps'$}
\end{align*}

\end{proof}

Now we lift the result of \cref{lem:TransientMD_nonuniform} from non-uniform
to uniform strategies (and to optimal strategies) and obtain the following theorem.
The proof is a generalization of a ``plastering'' construction by
Ornstein \cite{Ornstein:AMS1969}
(see also \cite{KMSTW2020}) from reachability to tail objectives,
which works by fixing MD strategies on ever expanding subsets of the state space.

\begin{restatable}{theorem}{thmOrnsteinplastering}\label{thm:Ornstein-plastering}
Let $\mdp=\mdptuple$ be a countable MDP, and let $\formula$ be an objective that is tail in~$\mdp$.
Suppose for every $s \in S$ there exist $\eps$-optimal MD strategies for~$\formula$.
Then:
\begin{enumerate}
\item There exist uniform $\eps$-optimal MD strategies for~$\formula$.
\item There exists a single MD strategy that is optimal from every state that has an optimal strategy.
\end{enumerate} 
\end{restatable}

\begin{restatable}{theorem}{thm:TransientMD}\label{thm:TransientMD}
In every countable MDP there exist uniform $\eps$-optimal MD strategies for
$\transient$.
Moreover, there exists a single MD strategy that is optimal for $\transient$ from every state that has an optimal strategy.
\end{restatable}
\begin{proof} 
Immediate from \cref{lem:TransientMD_nonuniform,thm:Ornstein-plastering},
since $\transient$ is a tail objective.
\end{proof}

%


\section{Strategy Complexity in Universally Transient MDPs}
\label{sec:parity}
The strategy complexity of parity objectives in general MDPs is known \cite{KMST2020c}.
Here we show that some parity objectives have a lower strategy complexity
in universally transient MDPs.
It is known \cite{KMST:ICALP2019}
that there are acyclic (and hence universally transient) MDPs
where $\eps$-optimal strategies for $\cParity{\{1,2\}}$
(and optimal strategies for $\cParity{\{1,2,3\}}$, resp.)
require $1$ bit.

We show that, for all simpler parity objectives in the Mostowski hierarchy \cite{Mostowski:84},
universally transient MDPs admit uniformly ($\eps$-)optimal MD strategies
(unlike general MDPs \cite{KMST2020c}).
These results (\cref{thm:012quant,thm:coBuchi}) ultimately rely on the existence of
uniformly $\eps$-optimal strategies for safety objectives.
While such strategies always exist for finitely branching MDPs -- simply pick a value-maximal successor --
this is not the case for infinitely branching MDPs \cite{KMSW2017}.
However, we show that universal transience implies the existence of 
uniformly $\eps$-optimal strategies for safety objectives even for \emph{infinitely branching} MDPs.

\newcommand{\revisit}[1]{{\it Re}(#1)} 
\newcommand{\NuOfRevisits}[1]{{\it R}(#1)} 

\begin{restatable}{theorem}{thmepsoptimalsafety}\label{thm:eps-optimal-safety}
For every universally transient countable MDP, safety objective and $\eps>0$
there exists a uniformly $\epsilon$-optimal MD strategy.
\end{restatable}

%
%
\begin{proof}
Let $\mdp=\mdptuple$ be a universally transient MDP and $\eps>0$.
Assume w.l.o.g.\ that the target $\reachset\subseteq\states$
of the objective $\formula=\safety{T}$
is a (losing) sink
and let $\iota:\states\to\N$ be an enumeration of the state space $\states$.

By \cref{lem:structural-transience}(3),
for every state $\state$ we have
$\revisit{\state} \eqdef
\sup_{\zstrat}\probm_{\mdp,\state,\zstrat}(\next\eventually(\state)) < 1$
and thus
$\NuOfRevisits{\state} \eqdef \sum_{i=0}^\infty \revisit{\state}^i < \infty$.
This means that, independent of the chosen strategy,
$\revisit{\state}$ upper-bounds the chance to return to $\state$,
and $\NuOfRevisits{\state}$ bounds the expected number of visits to $\state$.

Suppose that $\zstrat$ is an MD strategy which, at any state $\state\in\zstates$, picks a
successor $\state'$ with
$$\valueof{}{\state'}\quad\ge\quad \valueof{}{\state} -
\frac{\eps}{2^{\iota(\state)+1} \cdot \NuOfRevisits{\state}}.$$
This is possible even if $\mdp$ is infinitely branching, by the definition of
value and the fact that $\NuOfRevisits{\state} < \infty$.
We show that 
$\probm_{\mdp,\state_0,\zstrat}(\safety{\reachset}) \ge \valueof{}{\state_0} -\eps$
holds for every initial state $\state_0$, 
which implies the claim of the theorem.

\medskip
Towards this, we define a function $\cost$ that labels each transition in the MDP with a real-valued cost:
For every controlled transition $\state \transition \state'$ let
$\cost((\state,\state')) \eqdef \valueof{}{\state} - \valueof{}{\state'} \ge
0$.
Random transitions have cost zero.
We will argue that when playing $\zstrat$ from any start state $\state_0$, its
attainment w.r.t.\ the objective $\safety{\reachset}$
equals the value of $\state_0$ minus the expected total cost,
and that this cost is bounded by $\eps$.

For any $i\in\N$ let us write $s_i$ for the random variable denoting the state just after step $i$,
and $\costRV(i)\eqdef \cost(\state_{i},\state_{i+1})$ for the cost of step $i$ in a random run.
%
We observe that under $\zstrat$ the expected total cost is bounded in the
limit, i.e.,
\begin{equation}\label{eq:limcost-s}
    \lim_{n \to \infty} \expectation\left(\sum_{i=0}^{n-1}\costRV(i)\right) \le \eps.
\end{equation}
We moreover note that for every $n$,
\begin{equation}\label{eq:cost-s}
    \expectation(\valueof{}{s_n}) = \expectation(\valueof{}{s_0}) - \expectation\left(\sum_{i=0}^{n-1}\costRV(i)\right).
\end{equation}
Full proofs of the above two equations can be found in \cref{app-parity}.
Together they imply 
    \begin{equation}
        \label{eq:exi-lim-s}
        \liminf_{n\to\infty} \expectation(\valueof{}{\state_n})
        =
    \valueof{}{s_0} - \lim_{n \to \infty} \expectation\left(\sum_{i=0}^{n-1}\cost(i)\right)
        \ge \valueof{}{\state_0} - \eps.
    \end{equation}
Finally, to show the claim let $[\state_n\notin\reachset] : \states^\omega \to \{0,1\}$ be the random variable that indicates that the $n$-th state is not in the target set $\reachset$.
Note that $[\state_n\notin\reachset] \ge \valueof{}{\state_n}$ because target states have value $0$.
We have:
\begin{align*}
    \probm_{\mdp,s_0,\sigma}(\safety{\reachset})
 \quad&=\quad\probm_{\mdp,s_0,\sigma}\left(\bigwedge_{i=0}^\infty{\next^i \neg \reachset}\right)
 && \text{semantics of~$\safety{\reachset}=\always\neg \reachset$} \\
 & =\quad\smashoperator{\lim_{n\to\infty}} \probm_{\mdp,s_0,\sigma}\left(\bigwedge_{i=0}^n \next^i \neg \reachset \right)
 && \text{continuity of measures} \\
 & =\quad \smashoperator{\lim_{n\to\infty}} \probm_{\mdp,s_0,\sigma}(\next^n \neg \reachset)
 && \text{$\reachset$ is a sink} \\
 & =\quad\smashoperator{\lim_{n\to\infty}} \expectation([\state_n\notin\reachset])
 && \text{definition of $[\state_n\notin\reachset]$} \\
& \ge\quad \liminf_{n\to\infty} \expectation(\valueof{}{\state_n})
 && \text{as $[\state_n\notin\reachset] \ge \valueof{}{\state_n}$}\\
 & \ge\quad \valueof{}{\state_0}-\eps
 && \text{\cref{eq:exi-lim-s}.}
\qedhere
\end{align*}
\end{proof}

We can now combine \cref{thm:eps-optimal-safety} with the results from \cite{KMST2020c}
to show the existence of MD strategies assuming universal transience.

\begin{theorem}
\label{thm:012quant}
For universally transient MDPs 
optimal strategies for $\cParity{\{0,1,2\}}$, where they exist,
can be chosen uniformly MD.

Formally, let $\mdp$ be a universally transient MDP with states $\states$,
$\coloring:\states\to \{0,1,2\}$, and
$\formula=\Parity{\coloring}$.
There exists an MD strategy $\zstrat'$ that is optimal for all states $\state$
that have an optimal strategy:
$\big(
\exists \zstrat \in \zstratset.\,
\probm_{\mdp,\state,\zstrat}(\formula) = \valueof{\mdp}{s}
\big)
\implies
\probm_{\mdp,\state,\zstrat'}(\formula) = \valueof{\mdp}{s}
$.
\end{theorem}
\begin{proof}
    Let $\mdp_+$ be the conditioned version of $\mdp$ w.r.t.~$\formula$ (see \cite[Def.~19]{KMST2020c}
    for a precise definition).
    By \cref{lem:old-conditioned-MDP-preserves-structural-transience}, $\mdp_+$ is still a universally transient MDP
    and therefore by \cref{thm:eps-optimal-safety}, there exist uniformly
    $\eps$-optimal MD strategies for
    every safety objective and every $\eps>0$.
    The claim now follows from \cite[Theorem~22]{KMST2020c}.
\end{proof}

\begin{theorem}\label{thm:coBuchi}
For every universally transient countable MDP $\mdp$,
co-B\"uchi objective and $\eps>0$
there exists a uniformly $\eps$-optimal MD strategy.

Formally, let $\mdp$ be a universally transient countable MDP with states $\states$,
$\coloring:\states\to \{0,1\}$ be a coloring, $\formula=\Parity{\coloring}$ and $\eps>0$.

There exists an MD strategy $\zstrat'$ s.t.~for every state $\state$,
$
\probm_{\mdp,\state,\zstrat'}(\formula) \ge \valueof{\mdp}{s} - \eps
$.
\end{theorem}
\begin{proof}
    This directly follows from \cref{thm:eps-optimal-safety} and \cite[Theorem~25]{KMST2020c}.
\end{proof}

\section{The Conditioned MDP}\label{sec:conditioned}
Given an MDP~$\mdp$ and an objective~$\formula$ that is tail in~$\mdp$, a construction of a \emph{conditioned} MDP~$\mdp_+$ was provided in \cite[Lemma~6]{KMSW2017} that, very loosely speaking, ``scales up'' the probability of~$\formula$ so that any strategy $\zstrat$ is optimal in~$\mdp$ if it is almost surely winning in~$\mdp_+$.
For certain tail objectives, this construction was used in~\cite{KMSW2017} to reduce the sufficiency of MD strategies for \emph{optimal} strategies to the sufficiency of MD strategies for \emph{almost surely winning} strategies, which is a special case that may be easier to handle.

However, the construction was restricted to states that \emph{have} an optimal strategy.
In fact, states in~$\mdp$ that do not have an optimal strategy do not appear in~$\mdp_+$.
In the following, we lift this restriction by constructing a more general version of the conditioned MDP, called~$\pmdp$.
The MDP~$\pmdp$ will contain all states from~$\mdp$ that have a positive value w.r.t.~$\formula$ in~$\mdp$.
Moreover, all these states will have value~$1$ in~$\pmdp$.
It will then follow from \cref{lem:conditioned-construction}(3) below that an $\eps$-optimal strategy in~$\pmdp$ is $\varepsilon \valueof{\mdp}{s_0}$-optimal in~$\mdp$.
This allows us to reduce the sufficiency of MD strategies for $\eps$-optimal strategies to the sufficiency of MD strategies for $\eps$-optimal strategies for states with value~$1$.
In fact, it also follows that if an MD strategy~$\zstrat$ is uniform $\eps$-optimal in~$\pmdp$, it is \emph{multiplicatively} uniform $\eps$-optimal in~$\mdp$, i.e., $\probm_{\mdp,s,\zstrat}(\formula) \ge (1-\eps) \cdot \valueof{\mdp}{s}$ holds for all states~$s$.

\begin{definition}
\label{def:conditionedmdp}
For an MDP $\mdp=\mdptuple$ and an objective~$\formula$ that is tail in~$\mdp$,
define the \emph{conditioned version} of~$\mdp$ w.r.t.~$\formula$
to be the MDP $\pmdp = \tuple{\pstates,\pzstates,\prstates,\ptransition,\pprobp}$ with
\begin{align*}
\pzstates \ = \ &\{s \in \zstates \mid \valueof{\mdp}{s} > 0 \} \\
\prstates \ = \ &\{s \in \rstates \mid \valueof{\mdp}{s} > 0 \} \cup \{s_\bot\} \cup \{(s,t) \in \mathord{\transition} \mid s \in \zstates,\ \valueof{\mdp}{s} > 0\} \\
\ptransition \ = \ &\{(s,(s,t)) \in (\zstates \times \mathord\transition) \mid \valueof{\mdp}{s}>0,\ s \transition t\}  \cup \mbox{} \\
                   &\{(s,t) \in \rstates \times \states \mid \valueof{\mdp}{s}>0,\ \valueof{\mdp}{t}>0\} \cup \mbox{} \\
                   &\{((s,t),t) \in (\mathord\transition \times S) \mid \valueof{\mdp}{s}>0,\ \valueof{\mdp}{t}>0\} \cup \mbox{} \\
                   &\{((s,t),s_\bot) \in (\mathord\transition \times \{s_\bot\}) \mid \valueof{\mdp}{s}>\valueof{\mdp}{t}\} \cup \mbox{} \\
                   &\{(s_\bot,s_\bot)\} \\
\pprobp(s,t) \ = \ &\probp(s,t) \cdot \frac{\valueof{\mdp}{t}}{\valueof{\mdp}{s}} \hspace{20mm} 
\pprobp((s,t),t) \ = \ \frac{\valueof{\mdp}{t}}{\valueof{\mdp}{s}} \\
\pprobp((s,t),s_\bot) \ = \ & 1 - \frac{\valueof{\mdp}{t}}{\valueof{\mdp}{s}} \hspace{27mm} 
\pprobp(s_\bot,s_\bot) \ = \ 1
\end{align*}
for a fresh state $s_\bot$. 
\end{definition}
The conditioned MDP is well-defined.
Indeed, as $\formula$ is tail in~$\mdp$, for any $s \in \rstates$ we have $\valueof{\mdp}{s} = \sum_{s \transition t} \probp(s,t) \valueof{\mdp}{t}$, and so if $\valueof{\mdp}{s} > 0$ then $\sum_{s \transition t} \pprobp(s,t) = 1$.

\begin{restatable}{lemma}{lemconditionedconstruction}\label{lem:conditioned-construction}
Let $\mdp=\mdptuple$ be an MDP, and let $\formula$ be an objective that is tail in~$\mdp$.
Let $\pmdp = \tuple{\pstates,\pzstates,\prstates,\ptransition,\pprobp}$ be the conditioned version of~$\mdp$ w.r.t.~$\formula$.
Let $s_0 \in \pstates \cap \states$.
Let $\zstrat \in \zstratset_{\pmdp}$, and note that $\zstrat$ can be transformed to a strategy in~$\mdp$ in a natural way.
Then:
\begin{enumerate}
\item
For all $n \ge 0$ and all partial runs $s_0 s_1 \cdots s_n \in s_0 \pstates^*$ in~$\pmdp$ with $s_n \in \states$:
\[
\valueof{\mdp}{s_0} \cdot 
\probm_{\pmdp,s_0,\zstrat}(s_0 s_1 \cdots s_n \pstates^\omega) \ = \ 
\probm_{\mdp,s_0,\zstrat}(\overline{s_0 s_1 \cdots s_n} \states^\omega) \cdot \valueof{\mdp}{s_n}\,,
\]
where $\overline{w}$ for a partial run~$w$ in~$\pmdp$ refers to its natural contraction to a partial run in~$\mdp$; i.e., $\overline{w}$ is obtained from~$w$ by deleting all states of the form~$(s,t)$.
\item 
For all measurable $\playset \subseteq s_0 (\pstates \setminus \{s_\bot\})^\omega$ we have
\[
\probm_{\mdp,s_0,\zstrat}(\overline{\playset})
\ \ge \
\valueof{\mdp}{s_0} \cdot \probm_{\pmdp,s_0,\zstrat}(\playset) 
\ \ge \ 
\probm_{\mdp,s_0,\zstrat}(\overline{\playset} \cap \denotationof{\formula}{s_0})\,,
\]
where $\overline{\playset}$ is obtained from~$\playset$ by deleting, in all runs, all states of the form~$(s,t)$.
\item
We have
$\valueof{\mdp}{s_0} \cdot \probm_{\pmdp,s_0,\zstrat}(\formula) = \probm_{\mdp,s_0,\zstrat}(\formula)$.
In particular, $\valueof{\pmdp}{s_0} = 1$, and, for any $\varepsilon \ge 0$, strategy~$\sigma$ is $\varepsilon$-optimal in~$\pmdp$ if and only if it is $\varepsilon \valueof{\mdp}{s_0}$-optimal in~$\mdp$.
\end{enumerate}
\end{restatable}

\Cref{lem:conditioned-construction}.3 provides a way of proving the existence of MD strategies that attain, for each state~$s$, a fixed \emph{fraction} (arbitrarily close to~$1$) of the value of~$s$:

\begin{theorem} \label{thm:multiplicative-eps-optimal}
Let $\mdp=\mdptuple$ be an MDP, and let $\formula$ be an objective that is tail in~$\mdp$.
Let $\pmdp = \tuple{\pstates,\pzstates,\prstates,\ptransition,\pprobp}$ be the conditioned version of~$\mdp$ w.r.t.~$\formula$.
Let $\eps \ge 0$.
Any MD strategy~$\zstrat$ that is uniformly $\eps$-optimal in~$\pmdp$ (i.e., $\probm_{\pmdp,s,\zstrat}(\formula) \ge \valueof{\pmdp}{s} - \eps$ holds for all $s \in \pstates$) is \emph{multiplicatively $\eps$-optimal} in~$\mdp$ (i.e.,  $\probm_{\mdp,s,\zstrat}(\formula) \ge (1-\eps) \valueof{\mdp}{s}$ holds for all $s \in S$).
\end{theorem}
\begin{proof}
Immediate from \cref{lem:conditioned-construction}.3.
\end{proof}

As an application of \cref{thm:multiplicative-eps-optimal}, we can strengthen the first statement of \cref{thm:TransientMD} towards \emph{multiplicatively} (see \cref{thm:multiplicative-eps-optimal}) uniform $\eps$-optimal MD strategies for $\transient$.
\begin{corollary} \label{cor:multiplicative-eps-optimal-transience}
In every countable MDP there exist multiplicatively uniform $\eps$-optimal MD strategies for $\transient$.
\end{corollary}
\begin{proof}
Let $\mdp$ be a countable MDP, and $\pmdp$ its conditioned version w.r.t. $\transient$.
Let $\eps > 0$.
By \cref{thm:TransientMD}, there is a uniform $\eps$-optimal MD strategy~$\zstrat$ for $\transient$ in~$\pmdp$.
By \cref{thm:multiplicative-eps-optimal}, strategy~$\zstrat$ is multiplicatively uniform $\eps$-optimal in~$\mdp$.
\end{proof}

The following lemma, stating that universal transience is closed under ``conditioning'', is needed for the proof of \cref{lem:old-conditioned-MDP-preserves-structural-transience} below.

\begin{restatable}{lemma}{lemconditionedMDPpreservesstructuraltransience}\label{lem:conditioned-MDP-preserves-structural-transience}
Let $\mdp=\mdptuple$ be an MDP, and let $\formula$ be an objective that is tail in~$\mdp$.
Let $\pmdp = \tuple{\pstates,\pzstates,\prstates,\ptransition,\pprobp}$ be the conditioned version of~$\mdp$ w.r.t.~$\formula$, where $s_\bot$ is replaced by an infinite chain $s_\bot^1 \transition s_\bot^2 \transition \cdots$.
If $\mdp$ is universally transient, then so is~$\pmdp$.
\end{restatable}

In~\cite[Lemma~6]{KMSW2017} a variant, say~$\mdp_+$, of the conditioned MDP~$\pmdp$ from \cref{def:conditionedmdp} was proposed.
This variant~$\mdp_+$ differs from~$\pmdp$ in that $\mdp_+$ has only those states~$s$ from~$\mdp$ that have an optimal strategy, i.e., a strategy~$\zstrat$ with $\probm_{\mdp,s,\zstrat}(\formula) = \valueof{\mdp}{s}$.
Further, for any transition $s\transition t$ in~$\mdp_+$ where $s$ is a controlled state, we have $\valueof{\mdp}{s} = \valueof{\mdp}{t}$, i.e., $\mdp_+$ does not have value-decreasing transitions emanating from controlled states.
The following lemma was used in the proof of \cref{thm:012quant}:
\begin{restatable}{lemma}{lemoldconditionedMDPpreservesstructuraltransience} \label{lem:old-conditioned-MDP-preserves-structural-transience}
Let $\mdp$ be an MDP, and let $\formula$ be an objective that is tail in~$\mdp$.
Let $\mdp_+$ be the conditioned version w.r.t.~$\formula$ in the sense of~\cite[Lemma~6]{KMSW2017}.
If $\mdp$ is universally transient, then so is~$\mdp_+$.
\end{restatable}

\section{Conclusion}\label{sec:conclusion}
The $\transient$ objective admits $\eps$-optimal (resp.\ optimal)
MD strategies even in \emph{infinitely} branching MDPs.
This is unusual, since $\eps$-optimal strategies for most other objectives
require infinite memory if the MDP is infinitely branching
(in particular all objectives generalizing Safety \cite{KMSW2017}).

$\transient$ encodes a notion of continuous progress, which can be used as
a tool to reason about the strategy complexity of other objectives in
countable MDPs.
E.g., our result on $\transient$ is used 
in \cite{MM:CONCUR2021} as a building block to show upper
bounds on the strategy complexity of certain threshold objectives
w.r.t.\ mean payoff, total payoff and point payoff.

\newpage
\bibliography{journals,conferences,autocleaned}

\begin{thebibliography}{10}

\bibitem{NIPS2004_2569}
Pieter Abbeel and Andrew~Y. Ng.
\newblock Learning first-order {Markov} models for control.
\newblock In {\em Advances in Neural Information Processing Systems 17}. MIT
  Press, 2004.
\newblock URL:
  \url{http://papers.nips.cc/paper/2569-learning-first-order-markov-models-for-control}.

\bibitem{Flesch:JOTA2020}
Galit Ashkenazi-Golan, J\'{a}nos Flesch, Arkadi Predtetchinski, and Eilon
  Solan.
\newblock Reachability and safety objectives in {Markov} decision processes on
  long but finite horizons.
\newblock {\em Journal of Optimization Theory and Applications}, 2020.

\bibitem{ModCheckPrinciples08}
Christel Baier and Joost-Pieter Katoen.
\newblock {\em Principles of Model Checking}.
\newblock MIT Press, 2008.

\bibitem{billingsley-1995-probability}
Patrick Billingsley.
\newblock {\em Probability and Measure}.
\newblock Wiley, 1995.
\newblock Third Edition.

\bibitem{blondel2000survey}
Vincent~D. Blondel and John~N. Tsitsiklis.
\newblock A survey of computational complexity results in systems and control.
\newblock {\em Automatica}, 2000.

\bibitem{bauerle2011finance}
Nicole Bäuerle and Ulrich Rieder.
\newblock {\em Markov Decision Processes with Applications to Finance}.
\newblock Springer-Verlag Berlin Heidelberg, 2011.

\bibitem{chatterjee2012survey}
K.~Chatterjee and T.~Henzinger.
\newblock A survey of stochastic $\omega$-regular games.
\newblock {\em Journal of Computer and System Sciences}, 2012.

\bibitem{ModCheckHB18}
Edmund~M. Clarke, Thomas~A. Henzinger, Helmut Veith, and Roderick Bloem,
  editors.
\newblock {\em Handbook of Model Checking}.
\newblock Springer, 2018.
\newblock \href {http://dx.doi.org/10.1007/978-3-319-10575-8}
  {\path{doi:10.1007/978-3-319-10575-8}}.

\bibitem{CGP:book}
E.M. Clarke, O.~Grumberg, and D.~Peled.
\newblock {\em Model Checking}.
\newblock MIT Press, Dec. 1999.

\bibitem{Feller:book}
William Feller.
\newblock {\em An Introduction to Probability Theory and Its Applications}.
\newblock Wiley \& Sons, second edition, 1966.

\bibitem{FPS:2018}
J\'{a}nos Flesch, Arkadi Predtetchinski, and William Sudderth.
\newblock Simplifying optimal strategies in limsup and liminf stochastic games.
\newblock {\em Discrete Applied Mathematics}, 2018.

\bibitem{Hill-Pestien:1987}
T.P. Hill and V.C. Pestien.
\newblock The existence of good {Markov} strategies for decision processes with
  general payoffs.
\newblock {\em Stoch. Processes and Appl.}, 1987.

\bibitem{KMST:CONCUR2021}
S.~Kiefer, R.~Mayr, M.~Shirmohammadi, and P.~Totzke.
\newblock Transience in countable {MDPs}.
\newblock In {\em International Conference on Concurrency Theory}, LIPIcs,
  2021.
\newblock Full version at \url{https://arxiv.org/abs/2012.13739}.

\bibitem{KMST:ICALP2019}
Stefan Kiefer, Richard Mayr, Mahsa Shirmohammadi, and Patrick Totzke.
\newblock {B\"uchi} objectives in countable {MDPs}.
\newblock In {\em International Colloquium on Automata, Languages and
  Programming}, LIPIcs, 2019.
\newblock Full version at \url{https://arxiv.org/abs/1904.11573}.
\newblock \href {http://dx.doi.org/10.4230/LIPIcs.ICALP.2019.119}
  {\path{doi:10.4230/LIPIcs.ICALP.2019.119}}.

\bibitem{KMST2020c}
Stefan Kiefer, Richard Mayr, Mahsa Shirmohammadi, and Patrick Totzke.
\newblock {Strategy Complexity of Parity Objectives in Countable MDPs}.
\newblock In {\em International Conference on Concurrency Theory}, 2020.
\newblock \href {http://dx.doi.org/10.4230/LIPIcs.CONCUR.2020.7}
  {\path{doi:10.4230/LIPIcs.CONCUR.2020.7}}.

\bibitem{KMSTW2020}
Stefan Kiefer, Richard Mayr, Mahsa Shirmohammadi, Patrick Totzke, and Dominik
  Wojtczak.
\newblock How to play in infinite {MDP}s (invited talk).
\newblock In {\em International Colloquium on Automata, Languages and
  Programming}, 2020.
\newblock \href {http://dx.doi.org/10.4230/LIPIcs.ICALP.2020.3}
  {\path{doi:10.4230/LIPIcs.ICALP.2020.3}}.

\bibitem{KMSW2017}
Stefan Kiefer, Richard Mayr, Mahsa Shirmohammadi, and Dominik Wojtczak.
\newblock {Parity Objectives in Countable MDPs}.
\newblock In {\em Annual IEEE Symposium on Logic in Computer Science}, 2017.
\newblock \href {http://dx.doi.org/10.1109/LICS.2017.8005100}
  {\path{doi:10.1109/LICS.2017.8005100}}.

\bibitem{MM:CONCUR2021}
Richard Mayr and Eric Munday.
\newblock {Strategy Complexity of Mean Payoff, Total Payoff and Point Payoff
  Objectives in Countable MDPs}.
\newblock In {\em International Conference on Concurrency Theory}, LIPIcs,
  2021.
\newblock The full version is available on arXiv.

\bibitem{Mostowski:84}
A.~Mostowski.
\newblock Regular expressions for infinite trees and a standard form of
  automata.
\newblock In {\em Computation Theory}, LNCS, 1984.

\bibitem{Ornstein:AMS1969}
Donald Ornstein.
\newblock On the existence of stationary optimal strategies.
\newblock {\em Proceedings of the American Mathematical Society}, 1969.
\newblock \href {http://dx.doi.org/10.2307/2035700}
  {\path{doi:10.2307/2035700}}.

\bibitem{Puterman:book}
Martin~L. Puterman.
\newblock {\em Markov Decision Processes: Discrete Stochastic Dynamic
  Programming}.
\newblock John Wiley \& Sons, Inc., 1st edition, 1994.

\bibitem{Satayana:wiki}
George Santayana.
\newblock Reason in common sense.
\newblock In {\em Volume 1 of The Life of Reason}. 1905.
\newblock URL: \url{https://en.wikipedia.org/wiki/George_Santayana}.

\bibitem{schal2002markov}
Manfred Sch{\"a}l.
\newblock {Markov} decision processes in finance and dynamic options.
\newblock In {\em Handbook of {Markov} Decision Processes}. Springer, 2002.

\bibitem{sigaud2013markov}
Olivier Sigaud and Olivier Buffet.
\newblock {\em Markov Decision Processes in Artificial Intelligence}.
\newblock John Wiley \& Sons, 2013.

\bibitem{Sudderth:2020}
William~D. Sudderth.
\newblock Optimal {Markov} strategies.
\newblock {\em Decisions in Economics and Finance}, 2020.

\bibitem{sutton2018reinforcement}
R.S. Sutton and A.G Barto.
\newblock {\em Reinforcement Learning: An Introduction}.
\newblock Adaptive Computation and Machine Learning. MIT Press, 2018.

\bibitem{Vardi:probabilistic}
Moshe~Y. Vardi.
\newblock Automatic verification of probabilistic concurrent finite-state
  programs.
\newblock In {\em Annual Symposium on Foundations of Computer Science}. {IEEE}
  Computer Society, 1985.
\newblock \href {http://dx.doi.org/10.1109/SFCS.1985.12}
  {\path{doi:10.1109/SFCS.1985.12}}.

\end{thebibliography}

\newpage
\appendix

\section{Strategy Classes}\label{app-st-classes}
We formalize the amount of \emph{memory} needed to implement strategies.
Let $\memory$ be a countable set of memory modes. An \emph{update function}
is a function $\updatefun: \memory\times \states \to \dist(\memory\times \states)$
that meets the following two conditions, for all modes $\memconf \in \memory$:
\begin{itemize}
	\item for all controlled states~$\state\in \zstates$, 
            the distribution
            $\updatefun((\memconf,\state))$ is over 
	$\memory \times \{\state'\mid \state \transition{} \state'\}$.
	\item for all random states $\state \in \rstates$, we have that
            $\sum_{\memconf'\in \memory} \updatefun((\memconf,\state))(\memconf',\state')=P(\state)(\state')$.
\end{itemize}

\newcommand{\inducedStrat}[2]{#1[#2]}

An update function~$\updatefun$ together with an initial memory~$\memconf_0$
induce a strategy~$\inducedStrat{\updatefun}{\memconf_0}:\states^*\zstates \to \dist(S)$ as follows.
Consider the Markov chain with states set $\memory \times \states$,
transition relation $(\memory\x\states)^2$ and probability function~$\updatefun$.
Any partial run $\rho=s_0 \cdots s_i$ in $\mdp$ gives rise to a
set 
{$H(\rho)=\{(\memconf_0,s_0) \cdots (\memconf_i,s_i) \mid \memconf_0,\ldots, \memconf_i\in \memory\}$}
of partial runs in this Markov chain.
Each $\rho s \in \state_0 \states^{*} \zstates$
induces a 
probability distribution~$\mu_{\rho \state}\in \dist(\memory)$, 
the probability $\mu_{\rho \state}(\memconf)$ is 
the  probability of being in state $(\memconf,\state)$
conditioned on having taken some partial run
from~$H(\rho \state)$.
We define~$\inducedStrat{\updatefun}{\memconf_0}$ such that
$\inducedStrat{\updatefun}{\memconf_0}(\rho \state)(\state')\eqdef\sum_{\memconf,\memconf'\in \memory} \mu_{\rho \state}(\memconf) \updatefun((\memconf,\state))(\memconf',\state')$
for all $\rho \state\in \states^{*} \zstates$ and~$\state' \in \states$.

We say that a strategy $\zstrat$ can be \emph{implemented} with
memory~$\memory$ (and initial memory $\memconf_0$) if there exists 
an update function $\updatefun$ such that $\zstrat=\inducedStrat{\updatefun}{\memconf_0}$.
In this case we may also write $\inducedStrat{\zstrat}{\memconf_0}$ to explicitly specify the initial
memory mode $\memconf_0$. Based on this, we can define several classes of strategies:

 A strategy $\zstrat$ is \emph{memoryless}~(M) (also called \emph{positional}) 
if it can be implemented with a memory  of size~$1$. 
We may view
M-strategies as functions $\zstrat: \zstates \to \dist(\states)$.
A strategy~$\zstrat$ is \emph{finite memory}~(F) if 
there exists a finite memory~$\memory$ implementing~$\zstrat$.
More specifically, a strategy is \emph{$1$-bit} 
if it can be implemented with a memory of size~$2$.
Such a strategy is then determined by a function
$\updatefun:\{0,1\}\x \states \to \dist(\{0,1\} \x \states)$.
\emph{Deterministic 1-bit} strategies are  are both deterministic and 1-bit.

\section{Missing Proofs from \cref{sec:transientPre}} \label{app-transientPre}

In this section, we prove \cref{lem:reduction-finite-branch} from the main body.

\lemreductionfinitebranch*

\begin{proof}

Given an infinitely branching  MDP~$\mdp=\mdptuple$ with set~$S$ of states and an initial state~$s_0\in S$, we construct a finitely branching~$\mdp'$
with set~$S'$ of states such that $s_0\in S'$.
The reduction uses the concept of ``recurrent ladders''; see  Figure~\ref{fig:no-optimal}. 

\medskip

The reduction is as follows.
\begin{itemize}
  \item For all controlled state~$s$ in~$\mdp$ with infinite branching $s \to s_i$ for all $i\geq 1$, 
  we introduce a recurrent ladder in~$\mdp'$,  consisting the controlled states $(\ell_{s,i})_{i\in \mathbb{N}}$ and random states~$(\ell'_{s,i})_{i\geq 1}$.
 The set of transitions includes $s \to \ell_{s,0}$ and 
$\ell_{s,0} \to \ell_{s,1}$,  and for all $i\geq 1$ two transitions $\ell_{s,i} \to \ell'_{s,i}$, and $\ell_{s,i} \to s_i$. Moreover, $\ell'_{s,i} \step{\frac{1}{2}} \ell_{s,i+1}$ and $\ell'_{s,i} \step{\frac{1}{2}} \ell_{s,i-1}$. 
   Here, all states of the recurrent ladder are fresh states.
 
  \item For all random states $s$ in~$\mdp$ with infinite branching $s \step{p_i} s_i$ for all $i\geq 1$,  we use  a gadget $s \step{1}{} z_1, z_i \step{1-p_i'} z_{i+1}, z_i \step{p_i'} s_i$ for all
     $i \geq 1$, with fresh random states $z_i$ and suitably adjusted probabilities
    $p_i'$ to ensure that the gadget
    is left at state $s_i$ with exact probability~$p_i$, i.e.,
    $p_i' = p_i/(\prod_{j=1}^{i-1}(1-p_j'))$. 
\end{itemize}

See Figure~\ref{fig:reduction-inf} for a partial illustration.


Given $\rho=q_0 q_1 \cdots q_n\in S^+$
denote by $\mathrm{last}(\rho)=q_n$ the last state of $\rho$.

\bigskip

For the first item, let $\alpha_1$ be a general strategy $\alpha_1:\states^*\zstates \to \dist(S)$ in $\mdp$.
We define $\beta_1$ in~$\mdp'$ with the use of memory $\memory =\states^*\times \{\bot\}\cup\{i \in \mathbb{N}\mid i \leq 1\})$ and an update function~$u$; see 
Appendix~\ref{app-st-classes}. 
The definition of  $u:\memory \times S' \to \dist( \memory \times S')$ is as follows.
For all $q,q' \in S'$ and $\rho\in S^{*}$, 
\begin{itemize}
	\item for all $\memconf=(\rho,\bot)$ and $\memconf'=(\rho q,\bot)$, 
\[
  u(\memconf,q)(\memconf',q') =
  \begin{cases}
    \probp(q)(q') & \parbox{8cm}{if  $q\in \rstates$;}
  \\\\
      \alpha_1(\rho q)(q') &  \parbox{8cm}{if  $q\in \zstates$ is finitely branching in $\mdp$;}		
  \end{cases}
\]

	\item for all $\memconf=(\rho,\bot)$ and $\memconf'=(\rho q,j)$ with $j\geq 1$, 
\[
  u(\memconf,q)(\memconf',q') =
  \begin{cases}  
   \alpha_1(\rho q)(q_j) &  \parbox{8cm}{if  $q\in \zstates$ is infinitely branching in $\mdp$ with $q\to q_i$
    for all $i\geq 1$,  and  $q'=\ell_{q,0}$;}   		
  \end{cases}
\]

	\item for all $\memconf,\memconf'=(\rho,j)$  with $j$, 
\[
 u(\memconf,q)(\memconf,q') =
  \begin{cases}
   1 &  \parbox{8cm}{if $s=\mathrm{last}(\rho)$ was infinitely branching in~$\mdp$, and if    $q=\ell_{s,i}$, $q'=\ell'_{s,i}$ and $i<j$;} 		
  \end{cases}
\]

	\item for all $\memconf=(\rho,j)$ and $\memconf'=(\rho,\bot)$ with $j\geq 1$, 
\[
  u(\memconf,q)(\memconf,q') =
  \begin{cases}  
   1 &  \parbox{8cm}{if $s=\mathrm{last}(\rho)$ was infinitely branching in~$\mdp$ with $s\to s_i$
    for all $i\geq 1$, and if   $q=\ell_{s,j}$ and $q'=s_j$;}   		
  \end{cases}
\]

\item and $u(\memconf,q)(\memconf',q')=0$ otherwise. 
\end{itemize}

The strategy $\beta_1$ consists of the above update function~$u$ and 
 initial memory~$\memconf_0=(\epsilon,\bot)$ where $\epsilon$ is the empty run.
Intuitively speaking, in every step $\beta_1$ considers the memory~$(\rho,x)$  
and the current state~$q$ to simulate what $\alpha_1$ would have played in~$\mdp$.
The memory~$(\rho,x)$ is such that $\rho$ invariantly demonstrates the history of run projected into the state space~$S$ of~$\mdp$ (omitting the introduced states due to the reduction). 
The second component $x$ in the memory is $\bot$ if the current state is in $S$, 
and otherwise it is a natural number~$j \geq 1$. Such a natural number $j$ 
indicates that the controller is currently on a recurrent ladder  and must leaves the ladder at the $j$-th controlled state on the ladder. 
Subsequently, $\beta_1$ starts with memory $(\epsilon, \bot)$  and $q=s_0$, 
\begin{itemize}
	\item when $q$ is  a random state in $\mdp$, $\beta_1$ only append $q$ to $\rho$ to keep track of the history;
  \item when $q$ is  a finitely branching state in $\mdp$, $\beta_1$ plays as $\alpha_1(\rho q)$ and append $q$ to $\rho$;
	\item when $q$ is  an infinitely branching state in $\mdp$ with successors $(q_j)_{j\geq 1}$, for every $j\geq 1$,
	 the strategy $\beta_1$ chooses the first state~$\ell_{q,0}$ of the recurrent ladder for~$q$ 
	while flipping the memory from~$(\rho,\bot)$ to $(\rho,j)$ with  probability~$\sigma(\rho q)(q_j)$. This   
	requires the ladder to be traversed to state $\ell_{q,j}$ and left from there to $q_j$, 
	the $j$-th successor of $q$ in~$\mdp$. Furthermore, $\beta_1$ append $q$ to $\rho$;
	\item when $q$ is $\ell_{s,i}$ and memory is~$(\rho,j)$ with $\mathrm{last}(\rho)=s$, 
	if $i\leq j$ then $\beta_1$ continues to stay on the recurrent ladder by picking~$\ell'_{s,i}$;
	\item when $q$ is $\ell_{s,j}$ and memory is~$(\rho,j)$ with $\mathrm{last}(\rho)=s$, 
	$\beta_1$ leaves the ladder from $\ell_{s,j}$ to $q_j$ which is the $j$-th 
	successor of state~$s$ in $\mdp;$. In addition, the memory $(\rho,j)$ is flipped back to $(\rho,\bot)$.
\end{itemize}
By the construction of $\mdp'$ and $\beta_1$, 
it follows that $\beta_1$ in~$\mdp'$ faithfully simulates~$\alpha_1$ in~$\mdp$ and 
thus $\probm_{\mdp,s_0,\alpha_1}(\transient)=\probm_{\mdp',s_0,\beta_1}(\transient)$.

\bigskip

For the second item, let $\beta_2: \zstates' \to \zstates'$ be an MD strategy in 
$\mdp'$ where $\zstates' \subseteq S$ is the set of controlled states in $\mdp'$.
We define an MD strategy $\alpha_2:S \to S$ in $\mdp$ 
 as follows. For all controlled states $s\in S'$, 
\[
  \alpha_2(s) =
  \begin{cases}
    \beta_2(q) & \parbox{10.5cm}{if  $s\in \zstates$ is finitely branching in~$\mdp$;}
  \\\\
   s_j &  \parbox{10.5cm}{if  $s\in \zstates$ is infinitely branching in $\mdp$ with the successors~$(s_i)_{i\geq 1}$, and if there exists $j\in \mathbb{N}$ such that $
	\beta_2(s)=\ell_{s,0}$, $\beta_2(\ell_{s,0})=\ell_{s,1}$ and  $\beta_2(\ell_{s,i})=\ell'_{s,i}$ for all $0 < i<j$, and $\beta_2(\ell_{s,j})=s_j$;}	
		\\\\
		 s_1 &  \parbox{10.5cm}{if  $s\in \zstates$ is infinitely branching in $\mdp$ with the successors~$(s_i)_{i\geq 1}$, and if 
	$\beta_2(s)=\ell_{s,0}$,  $\beta_2(\ell_{s,0})=\ell_{s,1}$ and  $\beta_2(\ell_{s,i})=\ell'_{s,i}$ for all $i>0$.}		
  \end{cases}
\]
Note that the above strategy is well-defined, as in every recurrent ladder in~$\mdp'$, either there  exists some $j$
such that $\beta_2$ exits the ladder at its $j$-th controller state, or 
$\beta_2$ choose to stay on the ladder forever. In the latter case,  
by a Gambler's Ruin argument, the  probability of $\transient$ for those runs staying on the ladder forever  is $0$. 
By the construction of~$\mdp'$,
$\alpha_2$ faithfully simulates $\beta_2$ unless when $\beta_2$ stays on a ladder forever and  
the prospect of  $\transient$ becomes~$0$. In those cases, $\alpha_2$ continues playing what $\beta_2$
would have played if it exited the $s$-ladder at $\ell_{s,1}$. 

It follows that
 $\probm_{\mdp,s_0,\alpha_2}(\transient)\geq \probm_{\mdp',s_0,\beta_2}(\transient)$.
\end{proof}

\section{1-Bit Strategy for $\reset(F) \cap \transient$}\label{app:buchi-transient}
\thmMDPonebitBuchi*
\begin{proof}
We prove the claim for finitely branching $\mdp$ first and transfer the result to general MDPs at the end.

Let $\mdp=\mdptuple$ be a finitely branching countable MDP,
  $\initstates \subseteq \states$ a finite set of initial states and $F \subseteq \states$ a 
  set of goal states and $\formula \eqdef \reset(F) \cap \transient$
  the objective.

  \renewcommand{\optval}[1][\formula]{\valueof{\mdp,#1}{\state}}

For every $\eps >0$ and every $\state \in \initstates$ there exists an $\eps$-optimal strategy $\zstrat_\state$
such that
\begin{equation}\label{eq:eps-opt}
\probm_{\mdp,\state,\zstrat_\state}(\formula) \ge \optval-\eps.
\end{equation}
However, the strategies $\zstrat_\state$ might differ from each other and
might use randomization and a large (or even infinite)
amount of memory.
We will construct a single deterministic strategy $\zstrat'$ that uses only 1 bit of
memory such that $\forall_{\state \in \initstates}\, \probm_{\mdp,\state,\zstrat'}(\formula) \ge \optval-2\eps$.
This proves the claim as $\eps$ can be chosen arbitrarily small.

In order to construct $\zstrat'$, we first observe the behavior of the
finitely many $\zstrat_\state$ for $\state \in \initstates$
on an infinite, increasing sequence of finite subsets of $\states$.
Based on this, we define a second stronger objective $\formula'$ with
\begin{equation}\label{eq:prime-implies-normal}
\formula' \subseteq \formula,
\end{equation}
and show that all $\zstrat_\state$ 
attain at least $\optval-2\eps$
w.r.t.\ $\formula'$, i.e.,
\begin{equation}\label{eq:observe-orig}
    \forall_{\state \in \initstates}\, \probm_{\mdp,\state,\zstrat_\state}(\formula') \ge \optval-2\eps.
\end{equation}
We construct $\zstrat'$ as a deterministic 1-bit \emph{optimal} strategy w.r.t.\ $\formula'$
from all $\state \in \initstates$ and obtain
\begin{align*}
\probm_{\mdp,\state,\zstrat'}(\formula)\ &
\ge\ \probm_{\mdp,\state,\zstrat'}(\formula') && \text{by \cref{eq:prime-implies-normal}} \\
  & \ge\ \probm_{\mdp,\state,\zstrat_\state}(\formula') && \text{by optimality of $\zstrat'$ for $\formula'$}\\
  & \ge \ \optval-2\eps && \text{by \cref{eq:observe-orig}}.
\end{align*}

\medskip
\noindent
\smallparg{Behavior of $\zstrat$, objective $\formula'$ and properties
  \cref{eq:prime-implies-normal} and \cref{eq:observe-orig}.}
We start with some notation.
Let $\bubble{k}{X}$ be the set of states that can be reached from some state
in the set $X$ within at most $k$ steps.
Since $\mdp$ is finitely branching, $\bubble{k}{X}$ is finite if $X$ is
finite.
Let
$\setfb{\le k}{X} \eqdef\{\play \in \states^{\omega} \mid \exists t \le k.\, \play(t) \in X\}$
and
$\setfb{\ge k}{X} \eqdef\{\play \in \states^{\omega} \mid \exists t \ge k.\, \play(t) \in X\}$
denote the property of visiting the set $X$ (at least once) within at most
(resp.\ at least) $k$ steps.
Moreover, let $\eps_i \eqdef \eps \, \cdot\, 2^{-(i+1)}$.

\begin{figure*}[t]
   \begin{center}			
\centering

\begin{tikzpicture}[>=latex',shorten >=1pt,node distance=1.9cm,on grid,auto,
roundnode/.style={circle, draw,minimum size=1.5mm},
squarenode/.style={rectangle, draw,minimum size=2mm}]

\draw (3.8,0) ellipse (4.1cm and 1.3cm);
\draw [green!50!black,fill=green!40!white](-.3,0) arc (180:0:4.1cm and 1.3cm);
\draw [fill=white](2.8,0) ellipse (3.1cm and 1.1cm);
\draw (2.3,0) ellipse (2.6cm and .9cm);
\draw [green!50!black,fill=green!40!white](-.3,0) arc (180:0:2.6cm and .9cm);
\draw [fill=white](1.3,0) ellipse (1.6cm and .7cm);
\draw (.8,0) ellipse (1.1cm and .5cm);
\draw [green!50!black,fill=green!40!white](-.28,0) arc (180:0:1.09cm and .5cm);

\draw (11,1.39) arc (45:-45:3cm and 2cm);
\draw  (13.05,0.05) arc (0:-50:3cm and 2.2cm);

\draw [draw=none,name path=A] (13.05,0.05) arc (0:50:3cm and 2.2cm);
\draw [draw=none,name path=B](11,1.39) arc (45:0:3cm and 2cm);
\tikzfillbetween[of=A and B]{green!40!white};
\draw [green!50!black,name path=A] (13.05,0.05) arc (0:50:3cm and 2.2cm);
\draw [green!50!black,name path=B](11,1.39) arc (45:0:3cm and 2cm);

\draw (10,1.2) arc (40:-40:3cm and 1.8cm);
\draw[-,green!50!black] (0,0)--(1.95,0.05);
\draw[-,green!50!black] (2.9,0.02)--(4.95,0.05);
\draw[-,green!50!black](5.9,0.02)--(7.95,0.05);
\draw[-,green!50!black](11.87,-0.01)--(13.1,0.05);

\node [roundnode,fill=white] (s) at (0,0) {$I$};

\node[draw=none](dot1) at (9.2,0) {{\large $\cdots$}};
\node[draw=none](dot2) at (14,0) {{\large $\cdots$}};

\node[draw=none](K1) at (1.15,-.3)  {$K_1$};
\node[draw=none](L1) at (2.3,-.3)   {$L_1$};
\node[draw=none](K2) at (4.3,-.3)   {$K_2$};
\node[draw=none](L2) at (5.4,-.3)   {$L_2$};
\node[draw=none](K3) at (7.4,-.3)   {$K_3$};
\node[draw=none](Li) at (11.4,-.3)  {$L_{i-1}$};
\node[draw=none](Kii) at (12.6,-.3) {$K_i$};
\node[draw=none](Lii) at (13.6,-.3) {$L_i$};

\end{tikzpicture}
		\end{center}
	\caption{To show the bubble construction. The green region in $K_1$ is $F_1$, and for all $i\geq 2$, the green region in
	$K_{i}\setminus L_{i-1}$ is $F_i$.
}
\label{fig:KLBubbles}
\end{figure*}
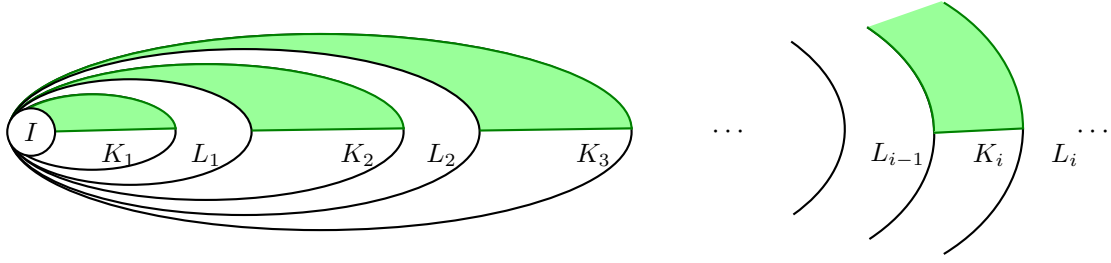

\begin{restatable}{lemma}{lembubbleextension}\label{lem-bubble-extension}
  Assume the setup of \cref{thm:MDP-one-bit-Buchi},
  $\formula \eqdef \reset(F) \cap \transient$ 
  and a strategy $\zstrat_\state$ from each $\state \in \initstates$.
  Let $X \subseteq \states$ be a finite set of states and $\eps' > 0$.
\begin{enumerate}
\item
There is $k \in \N$ such that
$
\forall_{\state\in\initstates}\,\probm_{\mdp,\state,\zstrat_\state}(\formula \cap 
\lnot(\setfb{\le k}{F \setminus X})
) \ \le \ \eps'.
$\label{lem-bubble-extension:ad1}
\item
There is $l \in \N$ such that
$\forall_{\state\in\initstates}\,
\probm_{\mdp,\state,\zstrat_\state}(\formula \cap \setfb{\ge l}{X}) \ \le \ \eps'$.
\label{lem-bubble-extension:ad2}
\end{enumerate}
\end{restatable}
\begin{proof}
It suffices to show the properties for a single $\state,\zstrat_\state$ since
one can take the maximal $k,l$ over the finitely many $\state \in \initstates$.

We observe that
$\formula \subseteq \transient = \bigcap_{\state \in \states} \eventually\always\neg(\state)
\subseteq \bigcap_{\state \in X} \eventually\always\neg(\state) =
\eventually\always\neg(X)$,
where the last equivalence is due to the finiteness of $X$.

Towards \ref{lem-bubble-extension:ad1}, we have
$\formula = \always\eventually F \cap \transient
\subseteq \always\eventually F \cap \eventually\always\neg(X)
\subseteq
\always\eventually (F \setminus X)
~\subseteq~ \eventually(F \setminus X) = \bigcup_{k \in \N} \setfb{\le k}{F \setminus X}$
and therefore that $\formula \cap \bigcap_{k \in \N}
\lnot(\setfb{\le k}{F \setminus X})
= \emptyset$.
It follows from the continuity of measures that $\lim_{k \to \infty}
\probm_{\mdp,\state,\zstrat_\state}(\formula \cap
\lnot(\setfb{\le k}{F \setminus X})
) = 0$.

Towards \ref{lem-bubble-extension:ad2},
we have
$\formula \cap \cap_l \setfb{\ge l}{X}
\subseteq 
\eventually\always\neg(X) \cap \cap_l \setfb{\ge l}{X}
=
\emptyset$.
By continuity of measures we obtain
$\lim_{l \rightarrow \infty} \probm_{\mdp,\state,\zstrat_\state}(\formula \cap \setfb{\ge l}{X})
= 0$.
\end{proof}

In the following, let us write $\complementof{X}$ to denote the complement of a set $X\subseteq S^\omega$ of runs.

By Lemma~\ref{lem-bubble-extension}(\ref{lem-bubble-extension:ad1}) there is a
$k_1$ such that for $K_1 \eqdef \bubble{k_1}{\initstates}$ and $F_1 \eqdef F \cap K_1$ we
have $\forall_{\state\in\initstates}\,\probm_{\mdp,\state,\zstrat_\state}(\formula \cap \complementof{K_1^* F_1 S^\omega}) \le \eps_1$.
We define the pattern
$$R_1 \eqdef (K_1 \setminus F_1)^*F_1$$
and obtain $\forall_{\state\in\initstates}\,\probm_{\mdp,\state,\zstrat_\state}(\formula \cap \complementof{R_1 S^\omega}) \le \eps_1$.
By Lemma~\ref{lem-bubble-extension}(\ref{lem-bubble-extension:ad2})
there is an $l_1 > k_1$ such that $\forall_{\state\in\initstates}\,\probm_{\mdp,\state,\zstrat_\state}(\setfb{\ge l_1}{K_1}) \le \eps_1$.
Define $L_1 \eqdef \bubble{l_1}{\initstates}$.
By Lemma~\ref{lem-bubble-extension}(\ref{lem-bubble-extension:ad1}) there is a
$k_2 > l_1$ such that for $K_2 \eqdef \bubble{k_2}{\initstates}$ and
$F_2 \eqdef F \cap K_2 \setminus L_1$ we have $\forall_{\state\in\initstates}\,\probm_{\mdp,\state,\zstrat_\state}(\formula \cap \complementof{K_2^* F_2 S^\omega}) \le \eps_2$.
We define the pattern
$$R_2 \eqdef (K_2 \setminus F_2)^*F_2$$
and obtain
$\forall_{\state\in\initstates}\,\probm_{\mdp,\state,\zstrat_\state}(\formula \cap \complementof{R_2
S^\omega}) \le \eps_2$ and, via a union bound,
$\forall_{\state\in\initstates}\,\probm_{\mdp,\state,\zstrat_\state}(\formula \cap \complementof{R_2 (S \setminus K_1)^\omega}) \le \eps_1 + \eps_2$.
By another union bound it follows that $\forall_{\state\in\initstates}\,\probm_{\mdp,\state,\zstrat_\state}(\formula \cap \complementof{R_1 R_2 (S \setminus K_1)^\omega}) \le 2 \eps_1 + \eps_2$.

Proceed inductively for $i = 2, 3, \ldots$ as follows (see \cref{fig:KLBubbles} for an illustration).
By Lemma~\ref{lem-bubble-extension}(\ref{lem-bubble-extension:ad2})
there is an $l_i > k_i$ such that $\forall_{\state\in\initstates}\,\probm_{\mdp,\state,\zstrat_\state}(\setfb{\ge l_i}{K_i}) \le \eps_i$.
Define $L_i \eqdef \bubble{l_i}{\initstates}$.
By Lemma~\ref{lem-bubble-extension}(\ref{lem-bubble-extension:ad1}) there is
$k_{i+1} > l_i$ such that for
$K_{i+1} \eqdef \bubble{k_{i+1}}{\initstates}$ and $F_{i+1} \eqdef F \cap
K_{i+1} \setminus L_i$ we have
$\forall_{\state\in\initstates}\,\probm_{\mdp,\state,\zstrat_\state}(\formula \cap \complementof{(K_{i+1} \setminus F_{i+1})^* F_{i+1} S^\omega}) \le \eps_{i+1}$.
By a union bound, $\forall_{\state\in\initstates}\,\probm_{\mdp,\state,\zstrat_\state}(\formula \cap \complementof{(K_{i+1} \setminus F_{i+1})^* F_{i+1} (S \setminus K_i)^\omega}) \le \eps_i + \eps_{i+1}$.
By an induction hypothesis we have $\forall_{\state\in\initstates}\,\probm_{\mdp,\state,\zstrat_\state}(\formula \cap \complementof{R_1 R_2 \ldots R_i (S \setminus K_{i-1})^\omega}) \le 2 \eps_1 + \cdots + 2 \eps_{i-1} + \eps_i$.
We define the pattern
$$R_{i+1} \eqdef (K_{i+1} \setminus (F_{i+1} \cup K_{i-1}))^*F_{i+1}.$$
Using that
$
 (K_{i+1} \setminus F_{i+1})^* F_{i+1} (S \setminus K_i)^\omega \ \cap \ R_1 R_2 \ldots R_i (S \setminus K_{i-1})^\omega
\subseteq
R_1 R_2 \ldots R_{i+1} (S \setminus K_i)^\omega,
 $
we get
\begin{equation}\label{eq:eps-bound}
    \forall_{\state\in\initstates}\,\probm_{\mdp,\state,\zstrat_\state}(\formula \cap \complementof{R_1 R_2 \ldots R_{i+1} (S \setminus K_i)^\omega})
\quad\le\quad
2 \eps_1 + \cdots + 2 \eps_{i} + \eps_{i+1}
\quad\le\quad
\eps.
\end{equation}

We now define the Borel objectives
$R_{\le i} \eqdef R_1 R_2 \dots R_i \states^\omega$ and
\[
\formula' \eqdef \bigcap_{i \in \N} R_{\le i}.
\]
Since $F_i \cap F_k = \emptyset$ for $i \neq k$ and $\formula'$ implies a visit
to the set $F_i$ for all $i \in \N$, we have
$\formula' \subseteq \reset(F)$.
Now we show that $\formula' \subseteq \transient$.
Let $\state$ be an arbitrary state and $\rho$ a run from some state in
$\initstates$ that satisfies $\formula'$.
If $\state$ is not reachable from $\initstates$ then $\rho$ never visits
$\state$.
Otherwise, there exists some minimal $j$ such that $\state \in K_j$.
The run $\rho$ must eventually visit $F_{j+1}$ and after visiting $F_{j+1}$
it cannot visit $K_j$ (and thus $\state$) any more. Therefore $\rho$ visits
$\state$ only finitely often.
Thus $\formula' \subseteq \transient$.
Together we have $\formula' \subseteq \reset(F) \cap \transient = \formula$ and obtain \cref{eq:prime-implies-normal}.

Moreover, $R_{\le 1} \supseteq R_{\le 2} \supseteq R_{\le 3} \dots$ is an
infinite decreasing sequence of Borel objectives.
For every $\state \in \initstates$ we have
\begin{align*}
\probm_{\mdp,\state,\zstrat_\state}(\formula')\ &
=\ \probm_{\mdp,\state,\zstrat_\state}(\cap_{i=1}^\infty R_{\le i}) && \text{by def.\ of $\formula'$}\\
& = \ \lim_{i \to \infty} \probm_{\mdp,\state,\zstrat_\state}(R_{\le i}) && \text{by cont.\ of measures}\\
& = \ \lim_{i \to \infty} 1-\probm_{\mdp,\state,\zstrat_\state}(\complementof{R_{\le i}}) && \text{by duality}\\
& = \ \lim_{i \to \infty} 1-(\probm_{\mdp,\state,\zstrat_\state}(\complementof{R_{\le i}}\cap
\formula) + \probm_{\mdp,\state,\zstrat_\state}(\complementof{R_{\le i}} \cap \complementof{\formula})) && \text{case split}\\
                                                                                                 & \ge \ \lim_{i \to \infty} 1-(\eps + \probm_{\mdp,\state,\zstrat_\state}(\complementof{R_{\le i}} \cap  \complementof{\formula})) &&
\text{by \cref{eq:eps-bound}}\\
& \ge \ \lim_{i \to \infty} 1-(\eps + \probm_{\mdp,\state,\zstrat_\state}(\complementof{\formula'}\cap \complementof{\formula})) &&
\text{since $\formula' \subseteq R_{\le i}$}\\
& = \ 1-(\eps + 1-\probm_{\mdp,\state,\zstrat_\state}(\formula' \cup \formula)) &&
\text{by duality}\\
& = \ \probm_{\mdp,\state,\zstrat_\state}(\formula) - \eps &&
\text{by \cref{eq:prime-implies-normal}}\\
& \ge \ \optval-2\eps && \text{by \cref{eq:eps-opt}}
\end{align*}
Thus we obtain property \cref{eq:observe-orig}.

\medskip
\noindent
\smallparg{Definition of the 1-bit strategy $\zstrat'$.}
We now define our deterministic 1-bit strategy $\zstrat'$ that is optimal for
objective $\formula'$ from \emph{every} $\state \in \initstates$.
First we define certain ``suffix'' objectives of $\formula'$.
Recall that $R_i = (K_{i} \setminus (F_{i} \cup K_{i-2}))^*F_{i}$.
Let $R_{i,j} \eqdef R_i R_{i+1} \dots R_j \states^\omega$
and $R_{\ge i} \eqdef \bigcap_{j \ge i} R_{i,j}$.
In particular, this means that $\formula' = R_{\ge 1}$.
Every run $w$ from some state $\state \in \initstates$ that satisfies $\formula'$ can be split into parts
before and after the first visit to set $F_i$, i.e.,
$w = w_1\state'w_2$ where $w_1\state' \in R_{\le i}$, $\state' \in F_i$ and
$\state' w_2 \in R_{\ge i+1}$.
(Note also that $w_2$ cannot visit any states in $K_{i-1}$.)
Thus it will be useful to consider the
objectives $R_{\ge i+1}$ for runs that start in states $\state' \in F_i$.
For every state $\state' \in F_i$ we consider its value w.r.t.\ the
objective $R_{\ge i+1}$, i.e.,
$\valueof{\mdp,R_{\ge i+1}}{\state'} \eqdef \sup_{\hat{\zstrat}} \probm_{\mdp,\state',\hat{\zstrat}}(R_{\ge i+1})$.

For every $i \ge 1$ we consider the finite subspace $K_i\setminus K_{i-2}$.
In particular, it contains the sets $F_{i-1}$ and $F_i$.
(For completeness let $K_0 \eqdef F_0 \eqdef \initstates$ and
$K_{-1} \eqdef \emptyset$.)
It is not enough to maximize the probability of reaching the set $F_i$ in each
$K_i$ individually. One also needs to maximize the potential of visiting
further sets $F_{i+1}, F_{i+2}, \dots$ in the indefinite future.
Thus we define the bounded total reward
objective $B_i$ for runs starting in $F_{i-1}$ as follows.
Runs that exit the subspace (either by leaving $K_i$ or by visiting $K_{i-2}$)
before visiting $F_i$ get reward $0$.
When some run reaches the set $F_i$ \emph{for the first time}
in some state $\state'$ then this run gets the reward of $\valueof{\mdp,R_{\ge i+1}}{\state'}$.
We can consider an induced finite MDP $\hat{\mdp}$ with state space $K_i\setminus K_{i-2}$,
plus a sink state (with reward $0$) that is reached immediately after visiting
any state in $F_i$ and whenever one exits the set $K_i\setminus K_{i-2}$.
In $\hat{\mdp}$ one gets a reward of
$\valueof{\mdp,R_{\ge i+1}}{\state'}$ for visiting $\state' \in F_i$ as above.
By \cite[Theorem 7.1.9]{Puterman:book}, there exists a uniform optimal MD strategy $\zstrat_i$ for this bounded total reward objective on the induced finite MDP $\hat{\mdp}$,
which can be directly applied for objective $B_i$ on the subspace $K_i\setminus K_{i-2}$ in $\mdp$.
(The strategy $\zstrat_i$ is not necessarily unique, but our results
hold regardless of which of them is picked.)

We now define $\zstrat'$ by combining
different MD strategies $\zstrat_i$, depending on the current state and on the
value of the 1-bit memory.
The intuition is that the strategy $\zstrat'$ has two modes: normal-mode
and next-mode.
In a state $\state' \in K_i \setminus K_{i-1}$, if the memory is $i\pmod 2$
then the strategy is in normal-mode and plays towards reaching $F_i$.
Otherwise, the strategy is in next-mode and plays towards reaching $F_{i+1}$
(normally this happens because $F_i$ has already been seen).

Initially $\zstrat'$ starts in a state $\state \in \initstates$ with the 1-bit memory set to
$1$.
We define the behavior of $\zstrat'$ in a state $\state' \in K_i \setminus K_{i-1}$
for every $i \ge 1$.
\begin{itemize}
\item
If the 1-bit memory is $i \pmod 2$ and $\state' \notin F_i$ then $\zstrat'$ plays like
$\zstrat_i$.
(Intuitively, one plays towards $F_i$, since one has not yet visited it.)
\item
If the 1-bit memory is $i \pmod 2$ and $\state' \in F_i$ then the 1-bit memory
is set to $(i+1) \pmod 2$, and $\zstrat'$ plays like $\zstrat_{i+1}$.
(Intuitively, one records the fact that one has already seen $F_i$ and then
targets the next set $F_{i+1}$.)
\item
If the 1-bit memory is $(i+1) \pmod 2$ then $\zstrat'$ plays like
$\zstrat_{i+1}$.
(Intuitively, one plays towards $F_{i+1}$, since one has already visited $F_i$.)
\end{itemize}
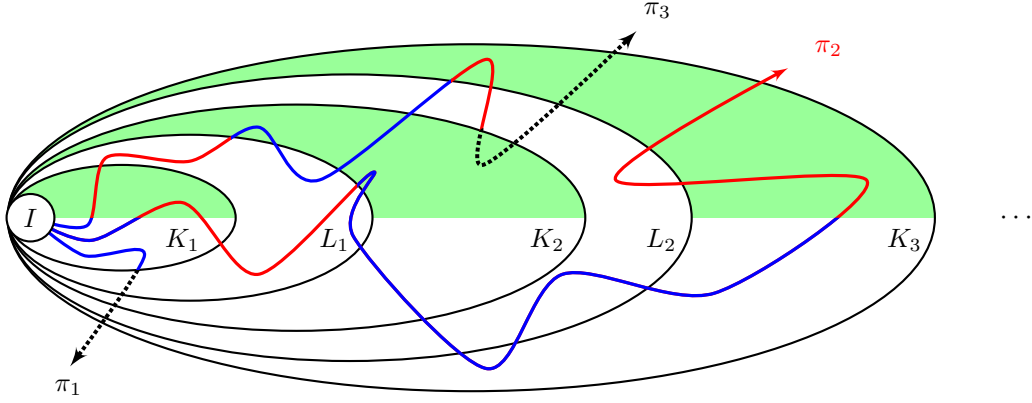
\begin{figure}[t]
   \begin{center}		
\centering

\makeatletter
\tikzset{
  use path for main/.code={%
    \tikz@addmode{%
      \expandafter\pgfsyssoftpath@setcurrentpath\csname tikz@intersect@path@name@#1\endcsname
    }%
  },
  use path for actions/.code={%
    \expandafter\def\expandafter\tikz@preactions\expandafter{\tikz@preactions\expandafter\let\expandafter\tikz@actions@path\csname tikz@intersect@path@name@#1\endcsname}%
  },
  use path/.style={%
    use path for main=#1,
    use path for actions=#1,
  }
}
\makeatother

\begin{tikzpicture}[>=latex',shorten >=1pt,node distance=1.9cm,on grid,auto,
fshade/.style={draw=none,fill=green!40!white,rotate=0},
bitone/.style={red,very thick},
bitzero/.style={blue, very thick},
lost/.style={densely dotted, ->, line width=1.5pt},
roundnode/.style={circle, draw,minimum size=1.5mm},
squarenode/.style={rectangle, draw,minimum size=2mm},
cross/.style={cross out,  fill=none, minimum size=2*(#1-\pgflinewidth), inner sep=0pt, outer sep=0pt}, cross/.default={1pt}]

\begin{scope}
\clip (5.8,0) ellipse (6.1cm and 2.3cm);
\draw[fshade] (-1,12) rectangle (12,0);
\end{scope}
\draw (5.8,0) ellipse (6.1cm and 2.3cm);

\draw [fill=white](4.2,0) ellipse (4.5cm and 1.9cm);

\begin{scope}
\clip (3.5,0) ellipse (3.8cm and 1.5cm);
\draw[fshade] (-1,10) rectangle (10,0);
\end{scope}
\draw (3.5,0) ellipse (3.8cm and 1.5cm);

\draw [fill=white](2.1,0) ellipse (2.4cm and 1.1cm);

\begin{scope}
\clip (1.2,0) ellipse (1.5cm and .7cm);
\draw[fshade] (-1,2) rectangle (3,0);
\end{scope}
\draw (1.2,0) ellipse (1.5cm and .7cm);


\node[draw=none](dot1) at (13,0)  {{\large $\cdots$}};
\node[draw=none](K1) at (2,-.3)   {$K_1$};
\node[draw=none](L1) at (4,-.3)   {$L_1$};
\node[draw=none](K2) at (6.8,-.3) {$K_2$};
\node[draw=none](L2) at (8.3,-.3) {$L_2$};
\node[draw=none](K3) at (11.5,-.3){$K_3$};

\coordinate (s) at (0,0);

\path[name path=firstrun] plot [smooth] coordinates {
    (s)
    (0.75,-.1)
    (1,0.8)
    (2.1,.75)
    (3,1.2)
    (3.8,.5)
    (6,2.1)
    (5.95,0.7)
    (8,2.5)
};
\begin{scope}
    \clip (s) rectangle (1,-0.5);
    \draw[bitzero,use path=firstrun];
\end{scope}
\begin{scope}
    \clip(s) rectangle (2.65,1.2);
    \draw[bitone,use path=firstrun];
\end{scope}
\begin{scope}
    \clip(2.65,-1.2) rectangle (5.54,2);
    \draw[bitzero,use path=firstrun];
\end{scope}
\begin{scope}
    \clip(5.54,2.5) rectangle (6.5,1.15);
    \draw[bitone,use path=firstrun];
\end{scope}
\begin{scope}
    \clip (5.9,1.15) -- (5,0) -- (7,0) -- (9,3) -- (8,3) --cycle;
    \draw[->,lost,use path=firstrun];
\end{scope}
\node[lost] at (8.25,2.75){$\pi_3$};

\draw[->,name path=secondrun,bitone] plot [smooth] coordinates {
    (s)
    (0.8,-.3)
    (2,.2)
    (3,-0.75)
    (4.5,0.6)
    (4.25,-0.25)
    (6,-2)
    (7,-0.75)
    (9,-1)
    (11,0.5)
    (7.7,.5)
    (10,2)
};
\begin{scope}
    \clip (0,0) -- (2,0) -- (1,-1.5) --cycle;
    \draw[bitzero,use path=secondrun];
\end{scope}
\begin{scope}
\clip(4.3,0.4) -- (4.25,1.5) -- (8,0) -- (10.9,0) -- (8,-2)-- (5,-2.1) -- (4,-0.5) --cycle;
    \draw[bitzero,use path=secondrun];
\end{scope}
\node[bitone] at (10.5,2.25){$\pi_2$};

\path [name path=thirdrun] plot [smooth] coordinates {
    (s)
    (0.7,-.5)
    (1.5,-0.5)
    (0.5,-2)
};
\begin{scope}
\clip (1.2,0) ellipse (1.5cm and .7cm);
    \draw[bitzero,use path=thirdrun];
\end{scope}
\begin{scope}
    \clip (0,-0.7) rectangle (2,-2);
    \draw[lost,->,use path=thirdrun];
\end{scope}
\node[lost] at (0.5,-2.25){$\pi_1$};

\node [roundnode,fill=white] at (s) {$I$};
\end{tikzpicture}
		\end{center}
                \caption{Memory updates along runs $\pi_1,\pi_2,\pi_3$,
                    drawn in blue while the memory-bit is one and in red while the bit is zero.
                    Both $\pi_1$ and $\pi_3$ violate $\varphi'$ and are drawn as dotted lines once they do.
}
\label{fig:flipingBit}
\end{figure}

Observe that if a run according to $\zstrat'$ exits some set $K_i$
(and thus enters $K_{i+1}\setminus K_i$) with the bit still set to $i \pmod 2$
(normal-mode) then this run has not visited $F_i$ and thus does not satisfy the objective
$\formula'$. (Or the same has happened earlier for some $j < i$, in which case
also the objective $\formula'$ is violated.)
An example is the run $\pi_1$ in \cref{fig:flipingBit}.

However, if a run according to $\zstrat'$ exits some set $K_i$
(and thus enters $K_{i+1} \setminus K_i$) with the bit set to $(i+1) \pmod 2$
(thus $\zstrat_{i+1}$ in next-mode)
then in the new set $K_{i'} \setminus K_{i'-1}$ with $i'=i+1$ the bit is set to
$i' \pmod 2$ and $\zstrat'$ continues to play like $\zstrat_{i+1}$ in normal-mode.
Even if this run returns (temporarily) to $K_i$ (but not to $K_{i-1}$)
the strategy $\zstrat'$ continues to play like $\zstrat_{i+1}$ in next-mode.
An example is the run $\pi_2$ in \cref{fig:flipingBit}.

Finally, if a run returns to $K_{i-1}$ after having visited $F_i$
then it fails the objective $\formula'$.
An example is the run $\pi_3$ in \cref{fig:flipingBit}.

\medskip
\noindent
\smallparg{The 1-bit strategy $\zstrat'$ is optimal for $\formula'$ from every
$\state \in \initstates$.}
In the following let $\state \in \initstates$ be an arbitrary initial state in $\initstates$.
For any run from $\state$, let $\firstinset{F_i}$ be the first state
$\state'$ in $F_i$ that is visited (if any).
We define a bounded reward objective $B_i'$ for runs starting
at $\state$ as follows. Every run that does not satisfy the objective
$R_{\le i}$ gets assigned reward $0$.
Otherwise, consider a run from $\state$ that satisfies $R_{\le i}$.
When this run reaches the set $F_i$ for the first time
in some state $\state'$ then this run gets a reward of
$\valueof{\mdp,R_{\ge i+1}}{\state'}$. Note that this reward is $\le 1$.

We show that for all $i \in \N$
\begin{equation}\label{eq:Bi-prime}
\valueof{\mdp,\formula'}{\state} = \valueof{\mdp,B_i'}{\state}
\end{equation}
Towards the $\ge$ inequality,
let $\hat{\zstrat}$ be an $\hat{\eps}$-optimal strategy for
$B_i'$ from $\state$.
We define the strategy $\hat{\zstrat}'$ to play like $\hat{\zstrat}$
until a state $\state' \in F_i$ is reached and then to switch to
some $\hat{\eps}$-optimal strategy for objective $R_{\ge i+1}$
from $\state'$.
Every run from $\state$ that satisfies $\formula'$ can be split into parts,
before and after the first visit to the set $F_i$, i.e.,
$\formula' = \{w_1\state'w_2\ |\ w_1\state' \in R_{\le i}, \state' \in F_i,
\state'w_2 \in R_{\ge i+1}\}$.
Therefore we obtain that
$\probm_{\mdp,\state,\hat{\zstrat}'}(\formula') \ge
\expectval_{\mdp,\state,\hat{\zstrat}}(B_i') - \hat{\eps} \ge
\valueof{\mdp,B_i'}{\state} - 2\hat{\eps}$.
Since this holds for every
$\hat{\eps} >0$, we obtain
$\valueof{\mdp,\formula'}{\state} \ge \valueof{\mdp,B_i'}{\state}$.

Towards the $\le$ inequality,
let $\hat{\zstrat}$ be any strategy for
$\formula'$ from $\state$. We have
$\probm_{\mdp,\state,\hat{\zstrat}}(\formula')
\le
\sum_{\state' \in F_i}
\probm_{\mdp,\state,\hat{\zstrat}}(R_{\le i} \cap \firstinset{F_i}=\state')
\cdot \valueof{\mdp,R_{\ge i+1}}{\state'}
=
\expectval_{\mdp,\state,\hat{\zstrat}}(B_i')
$.
Thus $\valueof{\mdp,\formula'}{\state} \le \valueof{\mdp,B_i'}{\state}$.
Together we obtain \cref{eq:Bi-prime}.

For all $i \in \N$ and every state $\state' \in F_i$ we show that
\begin{equation}\label{eq:Ri-eq-Bi}
\valueof{\mdp,R_{\ge i+1}}{\state'} = \valueof{\mdp,B_{i+1}}{\state'}
\end{equation}
Towards the $\ge$ inequality,
let $\hat{\zstrat}$ be an $\hat{\eps}$-optimal strategy for
$B_{i+1}$ from $\state' \in F_i$.
We define the strategy $\hat{\zstrat}'$ to play like $\hat{\zstrat}$
until a state $\state'' \in F_{i+1}$ is reached and then to switch to
some $\hat{\eps}$-optimal strategy for objective $R_{\ge i+2}$
from $\state''$. We have that
$\probm_{\mdp,\state',\hat{\zstrat}'}(R_{\ge i+1}) \ge
\expectval_{\mdp,\state',\hat{\zstrat}}(B_{i+1}) - \hat{\eps} \ge
\valueof{\mdp,B_{i+1}}{\state} - 2\hat{\eps}$.
Since this holds for every $\hat{\eps} >0$, we obtain
$\valueof{\mdp,R_{\ge i+1}}{\state'} \ge \valueof{\mdp,B_{i+1}}{\state'}$.

Towards the $\le$ inequality,
let $\hat{\zstrat}$ be any strategy for
$R_{\ge i+1}$ from $\state' \in F_i$.
We have
\begin{align*}
\probm_{\mdp,\state',\hat{\zstrat}}(R_{\ge i+1})
~&\le
\sum_{\state'' \in F_{i+1}}
\probm_{\mdp,\state',\hat{\zstrat}}(R_{i+1}\states^\omega \cap \firstinset{F_{i+1}}=\state'')
\cdot \valueof{\mdp,R_{\ge i+2}}{\state''}\\
 &=
\expectval_{\mdp,\state',\hat{\zstrat}}(B_{i+1}).
\end{align*}
Thus $\valueof{\mdp,R_{\ge i+1}}{\state'} \le \valueof{\mdp,B_{i+1}}{\state'}$.
Together we obtain \cref{eq:Ri-eq-Bi}.

We show, by induction on $i$, that $\zstrat'$ is optimal for $B_i'$ for
all $i \in \N$ from start state $\state$, i.e.,
\begin{equation}\label{eq:opt-Bi-prime}
\expectval_{\mdp,\state,\zstrat'}(B_i') = \valueof{\mdp,B_i'}{\state}
\end{equation}
In the base case of $i=1$ we have that $B_1' = B_1$. The strategy $\zstrat'$ plays
$\zstrat_1$ until reaching $F_1$, which is optimal for objective $B_1$ and
thus optimal for $B_1'$.
For the induction step we assume (IH) that $\zstrat'$ is optimal for $B_i'$.
\begin{align*}
\valueof{\mdp,B_{i+1}'}{\state}\ & =\ \valueof{\mdp,B_i'}{\state} && \text{by \cref{eq:Bi-prime}}\\
                           & =\ \expectval_{\mdp,\state,\zstrat'}(B_i') && \text{by (IH)}\\
& =\ \sum_{\state' \in F_i}
\probm_{\mdp,\state,\zstrat'}(R_{\le i} \cap \firstinset{F_i}=\state')
\cdot \valueof{\mdp,R_{\ge i+1}}{\state'} && \text{by def.\ of $B_i'$}\\
& =\ \sum_{\state' \in F_i}
\probm_{\mdp,\state,\zstrat'}(R_{\le i} \cap \firstinset{F_i}=\state')
\cdot \valueof{\mdp,B_{i+1}}{\state'} && \text{by \cref{eq:Ri-eq-Bi}}\\
& =\ \sum_{\state' \in F_i}
\probm_{\mdp,\state,\zstrat'}(R_{\le i} \cap \firstinset{F_i}=\state')
\cdot \expectval_{\mdp,\state',\zstrat_{i+1}}(B_{i+1}) && \text{opt.\ of $\zstrat_{i+1}$ for $B_{i+1}$}\\
& =\ \expectval_{\mdp,\state,\zstrat'}(B_{i+1}') && \text{by def.\ of
  $\zstrat'$ and $B_{i+1}'$}
\end{align*}
So $\zstrat'$ attains the value $\valueof{\mdp,B_{i+1}'}{\state}$ of the
objective $B_{i+1}'$ from $\state$ and is optimal. Thus \cref{eq:opt-Bi-prime}.

Now we show that $\zstrat'$ performs well on the objectives $R_{\le i}$ for
all $i \in \N$.
\begin{equation}\label{eq:1-bit-val}
\probm_{\mdp,\state,\zstrat'}(R_{\le i}) \ge \valueof{\mdp,\formula'}{\state}
\end{equation}
We have
\begin{align*}
\probm_{\mdp,\state,\zstrat'}(R_{\le i})\ &
\ge\ \expectval_{\mdp,\state,\zstrat'}(B_i') && \text{since $B_i'$ gives rewards
  $0$ for runs $\notin R_{\le i}$ and $\le 1$ otherwise} \\
  & =\ \valueof{\mdp,B_i'}{\state} && \text{by \cref{eq:opt-Bi-prime}}\\
  & = \  \valueof{\mdp,\formula'}{\state} && \text{by \cref{eq:Bi-prime}}
\end{align*}
So we get \cref{eq:1-bit-val}. Now we are ready to prove the optimality of $\zstrat'$ for $\formula'$ from $\state$.
\begin{align*}
\probm_{\mdp,\state,\zstrat'}(\formula')\ &
  =\ \probm_{\mdp,\state,\zstrat'}(\cap_{i \in \N} R_{\le i}) && \text{by def.\ of $\formula'$}\\
  & = \  \lim_{i \to \infty}\probm_{\mdp,\state,\zstrat'}(R_{\le i}) &&
\text{by continuity of measures from above}\\
  & \ge \ \lim_{i \to \infty}\valueof{\mdp,\formula'}{\state} && \text{by \cref{eq:1-bit-val}}\\
  & = \ \valueof{\mdp,\formula'}{\state}
\end{align*}
This concludes the proof that $\zstrat'$ is optimal for $\formula'$ and hence
$2\eps$-optimal for $\formula$ for every initial state
$\state \in \initstates$.

\medskip
\noindent
\smallparg{From finitely to infinitely branching MDPs.}
Let $\mdp$ be an infinitely branching MDP
with a finite set of initial states $\initstates$ and $\eps >0$. We derive a
finitely branching MDP $\mdp'$ with sufficiently similar behavior
wrt.\ our objective $\formula = \reset(F) \cap \transient$.
Every controlled state $x$ with infinite branching
$x \to y_i$ for all $i \in \N$ is replaced by a gadget
$x \to z_1, z_i \to z_{i+1}, z_i \to y_i$ for all $i \in \N$
with fresh controlled states $z_i$.
Infinitely branching random states with $x \step{p_i}{} y_i$ for all $i \in \N$
are replaced by a gadget
$x \step{1}{} z_1, z_i \step{1-p_i'} z_{i+1}, z_i \step{p_i'} y_i$ for all
$i \in \N$, with fresh random states $z_i$ and suitably adjusted probabilities
$p_i'$ to ensure that the gadget
is left at state $y_i$ with probability $p_i$, i.e.,
$p_i' = p_i/(\prod_{j=1}^{i-1}(1-p_j'))$.

We apply the above result for
finitely branching MDPs to $\mdp'$
and obtain a 1-bit deterministic $\eps$-optimal strategy $\zstrat'$
for our objective $\formula = \reset(F) \cap \transient$ from all states $\state \in \initstates$.
We construct a 1-bit deterministic $\eps$-optimal strategy $\zstrat''$
for $\mdp$ as follows.
Consider some state $x$ that is infinitely branching in $\mdp$
and its associated gadget in $\mdp'$.
Whenever a run in $\mdp'$ according to $\zstrat'$ reaches $x$ with some
memory value $\alpha \in \{0,1\}$ there exist
values $p_i$ for the probability that the gadget is left at state
$y_i$. Let $p \eqdef 1-\sum_{i\in \N} p_i$ be the probability that the gadget is
never left. (If $x$ is controlled then only one $p_i$ (or $p$) is nonzero,
since $\zstrat'$ is deterministic. If $x$ is random then $p=0$.)
Since $\zstrat'$ is deterministic, the memory updates are deterministic,
and thus there are values $\alpha_i' \in \{0,1\}$ such that whenever
the gadget is left at state $y_i$ the memory will be $\alpha_i'$.
We now define the behavior of the 1-bit deterministic strategy $\zstrat''$ at state $x$ with
memory $\alpha$ in $\mdp$.

If $x$ is controlled and $p\neq 1$ then $\zstrat''$ picks the successor state $y_i$ where
$p_i=1$ and sets the memory to $\alpha_i'$.
If $p=1$ then any run according to $\zstrat'$ that enters the gadget
does not satisfy the objective $\formula = \reset(F) \cap \transient$,
since the states in the gadget are disjoint from $F$. I.e., every run that
eventually stays in some gadget forever does not even satisfy $\reset(F)$, and
thus does not satisfy $\formula$.
Thus $\zstrat''$ performs at least as
well in $\mdp$ regardless of its choice, e.g., pick successor $y_1$ and
$\alpha' = \alpha$.

If $x$ is random then $p=0$ and the successor is chosen according to the defined
distribution (which is the same in $\mdp$ and $\mdp'$)
and $\zstrat''$ can only update its memory.
Whenever the successor $y_i$ is chosen, $\zstrat''$ updates the memory to
$\alpha_i'$.

In states that are not infinitely branching in $\mdp$, $\zstrat''$ does
exactly the same in $\mdp$ as $\zstrat'$ in $\mdp'$.

Since the gadgets do not intersect $F$, $\zstrat''$ performs
at least as well in $\mdp$ as $\zstrat'$ in $\mdp'$ and is thus
$\eps$-optimal from every $\state \in \initstates$.
\end{proof}

\begin{remark}
Note that the last step in the proof of \cref{thm:MDP-one-bit-Buchi},
lifting the result from finitely branching MDPs to infinitely branching MDPs,
does require this particular construction.
It cannot be shown by applying \cref{lem:reduction-finite-branch}.
The construction used for \cref{lem:reduction-finite-branch}
(i.e., Figure~\ref{fig:reduction-inf}) can only lift MD strategies,
but not deterministic 1-bit strategies.
The problem is that the construction in Figure~\ref{fig:reduction-inf}
introduces extra randomness and multiple paths to the same exit from the
ladder.
While an MD strategy on the finitely branching MDP $\mdp'$ induces
a corresponding MD strategy on the
infinitely branching MDP~$\mdp$, the same does not hold for
deterministic 1-bit strategies.
In contrast, the different construction in the last part of the proof
of \cref{thm:MDP-one-bit-Buchi}
preserves deterministic 1-bit strategies, but works only for the
$\reset(F) \cap \transient$ objective, not for $\transient$ alone.
\end{remark}

\section{Missing Proofs from \cref{transientMD}} \label{app-transientMD}
We prove \cref{thm:Ornstein-plastering} from the main body:
\thmOrnsteinplastering*

\subsection{Proof of Item~1 of \cref{thm:Ornstein-plastering}}

\begin{proof}
We follow Ornstein's proof~\cite{Ornstein:AMS1969} as presented in \cite{KMSTW2020}.
Recall that an MD strategy~$\sigma$ can be viewed as a function $\sigma : \zstates \to S$ such that for all $s \in \zstates$, the state~$\sigma(s)$ is a successor state of~$s$.
Starting from the original MDP~$\mdp$ we successively \emph{fix} more and more controlled states, by which we mean select an outgoing transition and remove all others.
While this is in general an infinite (but countable) process, it defines an MD strategy in the limit.
Visually, we ``plaster'' the whole state space by the fixings.

Put the states in some order, i.e., $s_1, s_2, \ldots$ with $S = \{s_1, s_2, \ldots\}$.
The plastering proceeds in \emph{rounds}, one round for every state.
Let $\mdp_i$ be the MDP obtained from~$\mdp$ after the fixings of the first $i-1$ rounds (with $\mdp_1 = \mdp$).
In round~$i$ we fix controlled states in such a way that
\begin{enumerate}
\item[(A)]
the probability, starting from~$s_i$, of~$\formula$ using only random and \emph{fixed} controlled states is not much less than the value $\valueof{\mdp_i}{s_i}$; and
\item[(B)] for all states~$s$, the value $\valueof{\mdp_{i+1}}{s}$ is almost as high as $\valueof{\mdp_{i}}{s}$.
\end{enumerate}
The purpose of goal~(A) is to guarantee good progress towards~$\formula$ when starting from~$s_i$.
The purpose of goal~(B) is to avoid fixings that would cause damage to the values of other states.
 

Now we describe round~$i$.
Consider the MDP~$\mdp_i$ after the fixings from the first $i-1$~rounds, and let $\varepsilon_i > 0$.
Recall that we wish to fix a part of the state space so that $s_i$ has a high probability of~$\formula$ using only random and fixed controlled states.
By assumption there is an MD strategy~$\sigma$ such that $\probm_{\mdp_i, s_i,\sigma}(\formula) \ge \valueof{\mdp_i}{s_i} - \varepsilon_i^2$.
Fixing~$\sigma$ everywhere would accomplish goal~(A), but potentially compromise goal~(B).
So instead we are going to fix~$\sigma$ only for states where $\sigma$~does well: define
\[
G \ \eqdef\ \{s \in S \mid \probm_{\mdp_i, s,\sigma}(\formula) \ge \valueof{\mdp_i}{s} - \varepsilon_i\}
\]
and obtain~$\mdp_{i+1}$ from~$\mdp_i$ by fixing~$\sigma$ on~$G$.
(Note that $\sigma$ does not ``contradict'' earlier fixings, because in the MDP~$\mdp_i$ the previously fixed states have only one outgoing transition left.)
%

We have to check that with this fixing we accomplish the two goals above.
Indeed, we accomplish goal~(A): by its definition strategy~$\sigma$ is $\varepsilon_i^2$-optimal from~$s_i$, so the probability of ever entering~$S\setminus G$ (where $\sigma$ is less than $\varepsilon_i$-optimal) cannot be large:
\begin{equation} \label{eq-goalA}
\probm_{\mdp_i, s_i,\sigma}(\reach{S \setminus G}) \ \le\ \varepsilon_i
\end{equation}
In slightly more detail, this inequality holds because the probability that the $\varepsilon_i^2$-optimal strategy~$\sigma$ enters a state whose value is underachieved by~$\sigma$ by at least~$\varepsilon_i$ can be at most~$\varepsilon_i$.
We give a detailed proof of \eqref{eq-goalA} in \cref{lem:app-proof-goalA} below.
It follows from the $\varepsilon_i^2$-optimality of~$\sigma$ and from~\eqref{eq-goalA} that we have $\probm_{\mdp_i, s_i,\sigma}(\formula \land \neg\reach{S \setminus G}) \ge \valueof{\mdp_i}{s_i} - \varepsilon_i - \varepsilon_i^2$.
So in~$\mdp_{i+1}$ we obtain for \emph{all} strategies~$\sigma'$:
\begin{equation} \label{eq-goalAA}
\probm_{\mdp_{i+1}, s_i,\sigma'}(\formula) \ \ge \ \valueof{\mdp_i}{s_i} - \varepsilon_i - \varepsilon_i^2 
\end{equation}

We also accomplish goal~(B): the difference between $\mdp_i$ and~$\mdp_{i+1}$ is that $\sigma$ is fixed on~$G$, but $\sigma$ performs well from $G$ on.
So we obtain for \emph{all} states~$s$:
\begin{equation} \label{eq-goalB}
\valueof{\mdp_{i+1}}{s} \ \ge \ \valueof{\mdp_{i}}{s} - \varepsilon_i
\end{equation}
In slightly more detail, this inequality holds because any strategy in~$\mdp_i$ can be transformed into a strategy in~$\mdp_{i+1}$, with the difference that once the newly fixed part $G$ is entered, the strategy switches to the strategy~$\sigma$, which (by the definition of~$\mdp_{i+1}$) is consistent with the fixing and (by the definition of~$G$) is $\varepsilon_i$-optimal from there.
We give a detailed proof of \eqref{eq-goalB} in \cref{lem:app-proof-goalB} below.
This completes the description of round~$i$.

Let $\varepsilon \in (0,1)$, and for all $i \ge 1$, choose $\varepsilon_i \eqdef \frac{\varepsilon}{2} \cdot 2^{-i}$.
Let $\sigma$ be an arbitrary MD strategy that is compatible with all fixings.
(This strategy~$\sigma$ is actually unique.)
It follows that $\sigma$ is playable in all~$\mdp_i$.
We have for all $i \ge 1$:
\begin{align*}
\probm_{\mdp, s_i,\sigma}(\formula)
\ &\ge \ \valueof{\mdp_i}{s_i} - \varepsilon_i - \varepsilon_i^2 && \text{by~\eqref{eq-goalAA}} \\
\ &\ge \ \valueof{\mdp_i}{s_i} - 2 \varepsilon_i && \text{as $\varepsilon_i < 1$} \\
\ &\ge \ \valueof{\mdp_i}{s_i} - \frac{\varepsilon}{2} && \text{choice of~$\varepsilon_i$} \\
\ &\ge \ \valueof{\mdp}{s_i} - \sum_{j=1}^{i-1} \varepsilon_j - \frac{\varepsilon}{2} && \text{by~\eqref{eq-goalB}} \\
\ &\ge \ \valueof{\mdp}{s_i} - \varepsilon && \text{choice of~$\varepsilon_j$}
\end{align*}
Thus, the MD strategy~$\sigma$ is $\varepsilon$-optimal for all states.
\end{proof}

\begin{lemma} \label{lem:app-proof-goalA}
\cref{eq-goalA} holds.
\end{lemma}
\begin{proof}
For a state $s \in S \setminus G$, define the event $L_s$ as the set of runs that leave~$G$ such that $s$ is the first visited state in~$S \setminus G$.
Then we have:
\begin{equation*} \label{eq-goalA-0}
\probm_{\mdp_i, s_i,\sigma}(\reach{S \setminus G}) \ =\ \sum_{s \in S \setminus G} \probm_{\mdp_i, s_i,\sigma}(L_s)
\end{equation*}
Since $\formula$ is tail and using the Markov property:
\begin{equation*} \label{eq-goalA-1}
\begin{aligned} 
\probm_{\mdp_i, s_i,\sigma}(\formula)
\ &=\ \probm_{\mdp_i, s_i,\sigma}(\neg\reach{S \setminus G} \land \formula) \ + \\
  &\qquad \sum_{s \in S \setminus G} \probm_{\mdp_i, s_i,\sigma}(L_s) \cdot \probm_{\mdp_i, s,\sigma}(\formula)
\end{aligned}
\end{equation*}
By the definition of~$G$ it follows:
\begin{equation} \label{eq-goalA-2}
\begin{aligned}
\probm_{\mdp_i, s_i,\sigma}(\formula)
\  &\le\ \probm_{\mdp_i, s_i,\sigma}(\neg\reach{S \setminus G} \land \formula) \ + \\
   &\qquad \sum_{s \in S \setminus G} \probm_{\mdp_i, s_i,\sigma}(L_s) \cdot (\valueof{\mdp_i}{s} - \varepsilon_i) 
\end{aligned}
\end{equation}
On the other hand, $\sigma$ is $\varepsilon_i^2$-optimal for~$s_i$, hence:
\begin{equation} \label{eq-goalA-3}
\begin{aligned}
\probm_{\mdp_i, s_i,\sigma}(\formula)
\ &\ge\ -\varepsilon_i^2 + \valueof{\mdp_i}{s} \\
\ &\ge\ -\varepsilon_i^2 + \probm_{\mdp_i, s_i,\sigma}(\formula \land \neg\reach{S \setminus G}) \ + \\
  &\qquad \sum_{s \in S \setminus G} \probm_{\mdp_i, s_i,\sigma}(L_s) \cdot \valueof{\mdp_i}{s}
\end{aligned}
\end{equation}
By combining \eqref{eq-goalA-2} and~\eqref{eq-goalA-3} we obtain:
\[
\varepsilon_i^2 \ \ge\ \varepsilon_i \cdot \sum_{s \in S \setminus G} \probm_{\mdp_i, s_i,\sigma}(L_s)\ = \ \varepsilon_i \cdot \probm_{\mdp_i, s_i,\sigma}(\reach{S \setminus G}) \tag*{\qedhere}
\]
\end{proof}

\begin{lemma} \label{lem:app-proof-goalB}
\cref{eq-goalB} holds.
\end{lemma}
\begin{proof}
For a state $s' \in G$, define the event $E_{s'}$ as the set of runs that enter~$G$ such that $s'$ is the first visited state in~$G$.
Fix any state $s \in S$ and any strategy~$\sigma_i$ in~$\mdp_i$.
We transform~$\sigma_i$ into a strategy~$\sigma_{i+1}$ in~$\mdp_{i+1}$ such that $\sigma_{i+1}$ behaves like~$\sigma_i$ until $G$ is entered, at which point $\sigma_{i+1}$ switches to the MD strategy~$\sigma$, which we recall is compatible with~$\mdp_{i+1}$ and is $\varepsilon_i$-optimal from~$G$ in~$\mdp_i$.
To show~\eqref{eq-goalB} it suffices to show that $\probm_{\mdp_{i+1},s,\sigma_{i+1}}(\formula) \ge \probm_{\mdp_{i},s,\sigma_{i}}(\formula) - \varepsilon_i$.
We have:
\begin{align*}
\probm_{\mdp_{i+1},s,\sigma_{i+1}}(\formula)
\ &=\ \probm_{\mdp_{i+1},s,\sigma_{i+1}}(\neg \reach{G} \land \formula) \ + &&\text{$\formula$ is tail}\\
  &\qquad \sum_{s' \in G} \probm_{\mdp_{i+1}, s,\sigma_{i+1}}(E_{s'}) \cdot \probm_{\mdp_{i+1}, s',\sigma_{i+1}}(\formula) &&\text{Markov property} \\
\ &=\ \probm_{\mdp_{i},s,\sigma_{i}}(\neg \reach{G} \land \formula) \ + && \text{using def.\ of~$\sigma_{i+1}$} \\
  &\qquad \sum_{s' \in G} \probm_{\mdp_{i}, s,\sigma_{i}}(E_{s'}) \cdot \probm_{\mdp_{i}, s',\sigma}(\formula)
\intertext{Further we have for all $s' \in G$:}
\probm_{\mdp_{i}, s',\sigma}(\formula)
\ &\ge\ \valueof{\mdp_{i}}{s'} - \varepsilon_i && \text{as $s' \in G$} \\
\ &\ge\ \probm_{\mdp_{i}, s',\sigma_i}(\formula) - \varepsilon_i
\intertext{Plugging this in above, we obtain:}
\probm_{\mdp_{i+1},s,\sigma_{i+1}}(\formula)
\ &\ge\ \probm_{\mdp_{i},s,\sigma_{i}}(\neg \reach{G} \land \formula) \ + \\
  &\qquad \sum_{s' \in G} \probm_{\mdp_{i}, s,\sigma_{i}}(E_{s'}) \cdot (\probm_{\mdp_{i}, s',\sigma_i}(\formula) - \varepsilon_i) \\
\ &\ge\ \probm_{\mdp_{i},s,\sigma_{i}}(\neg \reach{G} \land \formula) \ + \\
  &\qquad \Big(\sum_{s' \in G} \probm_{\mdp_{i}, s,\sigma_{i}}(E_{s'}) \cdot \probm_{\mdp_{i}, s',\sigma_i}(\formula)\Big) - \varepsilon_i \\
\ &=\ \probm_{\mdp_{i},s,\sigma_{i}}(\formula) - \varepsilon_i \tag*{\qedhere}
\end{align*}
\end{proof}

\subsection{Proof of Item~2 of \cref{thm:Ornstein-plastering}}
\begin{proof}
As discussed in \cref{sec:conditioned}, in \cite[Lemma~6]{KMSW2017} there is a construction of a certain \emph{conditioned version} of~$\mdp$ (similar to~$\pmdp$ from \cref{def:conditionedmdp}), say $\mdp_+$.
The construction is such that $\formula$ is tail also in~$\mdp_+$.
By \cite[Lemma~6, item~2]{KMSW2017} it suffices to exhibit a single MD strategy in~$\mdp_+$ that is \emph{almost surely winning} from all states that have an \emph{almost surely winning} strategy.

Obtain from~$\mdp_+$ an MDP~$\mdp'$ by restricting the state space to those states that have an almost surely winning strategy, and eliminating all transitions leaving these states.
In $\mdp'$ all states have an almost surely winning strategy, as an almost surely winning strategy may never enter a state that does not have an almost surely winning strategy (using the fact that $\formula$ is tail).
Let $\zstrat$ be a uniform $\frac12$-optimal MD strategy (in~$\mdp'$), which exists by item~1.
It suffices to show that $\zstrat$ is (in~$\mdp'$) almost surely winning from all states that have an almost surely winning strategy.

We follow the argument from \cite[Theorem~6]{KMSTW2020}.
We have $\probm_{\mdp',s,\sigma}(\formula) \ge \frac12$ for all states~$s$.
Thus, for any run $s_0 s_1 \cdots$ in~$\mdp'$ we have $\probm_{\mdp',s_i,\sigma}(\neg\formula) \le \frac12$ for all~$i$; in particular, the sequence $\left(\probm_{\mdp',s_i,\sigma}(\neg\formula)\right)_i$ does not converge to~$1$.
As a consequence of L\'evy's zero-one law, since $\neg\formula$ is tail, the events $\neg\formula$ and $\left\{ s_0 s_1 \cdots \;\middle\vert\; \lim_{i \to \infty} \probm_{\mdp',s_i,\sigma}(\neg\formula) = 1 \right\}$ are equal up to a null set.
Thus, for all states~$s$ we have $\probm_{\mdp',s,\sigma}(\neg\formula) = 0$; hence, $\probm_{\mdp',s,\sigma}(\formula) = 1$.
\end{proof}

\section{Missing Proofs from \cref{sec:parity}} \label{app-parity}
\thmepsoptimalsafety*
\begin{proof}
    Let $\mdp=\mdptuple$ be a universally transient MDP and $\eps>0$.
Assume w.l.o.g.\ that the target $\reachset\subseteq\states$
of the objective $\formula=\safety{T}$
is a (losing) sink
and let $\iota:\states\to\N$ be an enumeration of the state space $\states$.

By \cref{lem:structural-transience}(3),
for every state $\state$ we have
$\revisit{\state} \eqdef
\sup_{\zstrat}\probm_{\mdp,\state,\zstrat}(\next\eventually(\state)) < 1$
and thus
$\NuOfRevisits{\state} \eqdef \sum_{i=0}^\infty \revisit{\state}^i < \infty$.
This means that, independent of the chosen strategy,
$\revisit{\state}$ upper-bounds the chance to return to $\state$,
and $\NuOfRevisits{\state}$ bounds the expected number of visits to $\state$.

Suppose that $\zstrat$ is an MD strategy which, at any state $\state\in\zstates$, picks a
successor $\state'$ with
$$\valueof{}{\state'}\quad\ge\quad \valueof{}{\state} -
\frac{\eps}{2^{\iota(\state)+1} \cdot \NuOfRevisits{\state}}.$$
This is possible even if $\mdp$ is infinitely branching, by the definition of
value and the fact that $\NuOfRevisits{\state} < \infty$.
We show that 
$\probm_{\mdp,\state_0,\zstrat}(\safety{\reachset}) \ge \valueof{}{\state_0} -\eps$
holds for every initial state $\state_0$, 
which implies the claim of the theorem.

Towards this, we define a function $\cost$ that labels each transition in the MDP with a real-valued cost:
For every controlled transition $\state \transition \state'$ let
$\cost((\state,\state')) \eqdef \valueof{}{\state} - \valueof{}{\state'} \ge
0$.
Random transitions have cost zero.
We will argue that when playing $\zstrat$ from any start state $\state_0$, its
attainment w.r.t.\ the objective $\safety{\reachset}$
equals the value of $\state_0$ minus the expected total cost,
and that this cost is bounded by $\eps$.

For any $i\in\N$ let us write $s_i$ for the random variable denoting the state just after step $i$,
and $\costRV(i)\eqdef \cost(\state_{i},\state_{i+1})$ for the cost of step $i$ in a random run.
%
We now show that under $\zstrat$ the expected total cost is bounded in the
limit, i.e.,
\begin{equation}\label{eq:limcost}
    \lim_{n \to \infty} \expectation\left(\sum_{i=0}^{n-1}\costRV(i)\right) \le \eps.
\end{equation}
To show this, let us decompose the cost function as $\cost = \sum_\state \cost_\state$
where $\cost_\state$ is a local cost function for state $\state$ that assigns
$\valueof{}{\state} - \valueof{}{\state'}$ to all controlled transitions 
$\state \transition \state'$ \emph{starting in $\state$} and zero otherwise.
Similarly, we let $\costRV(n) \eqdef \sum_\state \costRV_\state(n)$,
where $\costRV_\state(n)$ is the random variable
denoting the cost incurred on in step $n$ from $\state$.
We thus have 
\begin{equation*}
\lim_{n \to \infty} \expectation\left(\sum_{i=0}^{n-1}\costRV(i)\right)
=
\lim_{n \to \infty} \expectation\left(\sum_{\state\in\states} \sum_{i=0}^{n-1}\costRV_s(i)\right)
=
\sum_{\state\in\states} \lim_{n \to \infty} \expectation\left(\sum_{i=0}^{n-1}\costRV_s(i)\right)
\end{equation*}
where the last equality holds by convergence of monotone series.

We now show an upper bound on $\lim_{n \to \infty} \expectation\left(\sum_{i=0}^{n-1}\costRV_\state(i)\right)$
for some fixed state $\state$.
Costs are only incurred at state $\state$, and each time they are upper-bounded by 
$\frac{\eps}{2^{\iota(\state)+1} \cdot \NuOfRevisits{\state}}$. Moreover, the probability of
returning from $\state$ to $\state$ is upper-bounded by $\revisit{\state}$. This means that
$   %
\lim_{n \to \infty} \expectation\left(\sum_{i=0}^{n-1}\costRV_\state(i)\right)
\le
\sum_{i=0}^\infty \revisit{\state}^i \frac{\eps}{2^{\iota(\state)+1} \cdot \NuOfRevisits{\state}}
=
\frac{\eps}{2^{\iota(\state)+1}}
$,
which in turn implies Eq.~\eqref{eq:limcost} as then
$\lim_{n \to \infty} \expectation\left(\sum_{i=0}^{n-1}\costRV(i)\right) =
\sum_{\state\in\states} \lim_{n \to \infty} \expectation\left(\sum_{i=0}^{n-1}\costRV_\state(i)\right)
\le \sum_\state \frac{\eps}{2^{\iota(\state)+1}} = \eps$.

Next, we show that for every $n$,
\begin{equation}\label{eq:cost}
    \expectation(\valueof{}{s_n}) = \expectation(\valueof{}{s_0}) - \expectation\left(\sum_{i=0}^{n-1}\costRV(i)\right).
\end{equation}
By induction on $n$ where the base case $n=0$ trivially holds.
For the induction step,
\begin{align*}
\expectation(\valueof{}{\state_{n+1}})
\quad&=\quad \expectation(\valueof{}{\state_{n}} + \valueof{}{\state_{n+1}} - \valueof{}{\state_{n}})\\
\quad&=\quad \expectation(\valueof{}{\state_{n}}) + \expectation(\valueof{}{\state_{n+1}} - \valueof{}{\state_{n}})\\
\quad&=\quad \expectation(\valueof{}{\state_{n}}) + \probm(s_n\in \rstates)\expectation(\valueof{}{\state_{n+1}} - \valueof{}{\state_{n}} \mid s_n \in \rstates)\\
     &\qquad\qquad\qquad~~  + \probm(s_n\in \zstates)\expectation(\valueof{}{\state_{n+1}} - \valueof{}{\state_{n}} \mid s_n \in \zstates)\\
\quad&=\quad \expectation(\valueof{}{\state_{n}}) + 0 - \probm(s_n\in \zstates)\expectation(\costRV({n}) \mid s_n \in \zstates)\\
\quad&=\quad \expectation(\valueof{}{\state_{n}}) - \probm(s_n\in \rstates)\expectation(\costRV({n}) \mid s_n \in \rstates) \\ 
     &\qquad\qquad\qquad~~  - \probm(s_n\in \zstates)\expectation(\costRV({n}) \mid s_n \in \zstates)\\
\quad&=\quad \expectation(\valueof{}{\state_{n}}) - \expectation(\costRV({n}))\\
\quad&=\quad \expectation(\valueof{}{\state_{0}}) - \sum_{i=0}^{n-1} \expectation(\costRV({i})) - \expectation(\costRV({n}))\\
\quad&=\quad \expectation(\valueof{}{\state_{0}}) - \expectation\left(\sum_{i=0}^{n} \costRV({i})\right).
\end{align*}

From \cref{eq:limcost,eq:cost} 
we get
    \begin{equation}
        \label{eq:exi-lim}
        \liminf_{n\to\infty} \expectation(\valueof{}{\state_n})
        =
    \valueof{}{s_0} - \lim_{n \to \infty} \expectation\left(\sum_{i=0}^{n-1}\cost(i)\right)
        \ge \valueof{}{\state_0} - \eps.
    \end{equation}
Finally, to show the claim let $[\state_n\notin\reachset] : \states^\omega \to \{0,1\}$ be the random variable that indicates that the $n$-th state is not in the target set $\reachset$.
Note that $[\state_n\notin\reachset] \ge \valueof{}{\state_n}$ because target states have value $0$.
We have:
\begin{align*}
    \probm_{\mdp,s_0,\sigma}(\safety{\reachset})
 \quad&=\quad\probm_{\mdp,s_0,\sigma}\left(\bigwedge_{i=0}^\infty{\next^i \neg \reachset}\right)
 && \text{semantics of~$\safety{\reachset}=\always\neg \reachset$} \\
 & =\quad\smashoperator{\lim_{n\to\infty}} \probm_{\mdp,s_0,\sigma}\left(\bigwedge_{i=0}^n \next^i \neg \reachset \right)
 && \text{continuity of measures} \\
 & =\quad \smashoperator{\lim_{n\to\infty}} \probm_{\mdp,s_0,\sigma}(\next^n \neg \reachset)
 && \text{$\reachset$ is a sink} \\
 & =\quad\smashoperator{\lim_{n\to\infty}} \expectation([\state_n\notin\reachset])
 && \text{definition of $[\state_n\notin\reachset]$} \\
& \ge\quad \liminf_{n\to\infty} \expectation(\valueof{}{\state_n})
 && \text{as $[\state_n\notin\reachset] \ge \valueof{}{\state_n}$}\\
 & \ge\quad \valueof{}{\state_0}-\eps
 && \text{\cref{eq:exi-lim}.}
\qedhere
\end{align*}
\end{proof}

\section{Missing Proofs from \cref{sec:conditioned}} \label{app-conditioned}
We prove \cref{lem:conditioned-construction} from the main body:

\lemconditionedconstruction*

\begin{proof}
We prove the equality in item~1 by induction on~$n$.
For $n=0$ it is trivial.
For the step, suppose the equality holds for some~$n$.
Let $s_0 s_1 \cdots s_n \in s_0 \pstates^*$ be a partial run in~$\pmdp$ with $s_n \in \states$.

Let $s_n \in \zstates$ and $s_{n+1} \in \pstates \cap \states$.
We have:
\begin{align*}
& \valueof{\mdp}{s_0} \cdot \probm_{\pmdp,s_0,\zstrat}(s_0 s_1 \cdots s_n (s_n,s_{n+1}) s_{n+1} \pstates^\omega) \\
& = \valueof{\mdp}{s_0} \cdot \probm_{\pmdp,s_0,\zstrat}(s_0 s_1 \cdots s_n \pstates^\omega) \cdot \zstrat(s_0 s_1 \ldots s_n)((s_n,s_{n+1})) \cdot \frac{\valueof{\mdp}{s_{n+1}}}{\valueof{\mdp}{s_n}} && \text{def.~of~$\pprobp$} \\
& = \probm_{\mdp,s_0,\zstrat}(\overline{s_0 s_1 \cdots s_n} \states^\omega) \cdot \zstrat(s_0 s_1 \ldots s_n)((s_n,s_{n+1})) \cdot \valueof{\mdp}{s_{n+1}} && \text{ind.\ hyp.} \\
& = \probm_{\mdp,s_0,\zstrat}(\overline{s_0 s_1 \cdots s_n} \states^\omega) \cdot \zstrat(\overline{s_0 s_1 \ldots s_n})(s_{n+1}) \cdot \valueof{\mdp}{s_{n+1}} && \text{$\zstrat$ in~$\mdp$} \\
& = \probm_{\mdp,s_0,\zstrat}(\overline{s_0 s_1 \cdots s_n (s_n s_{n+1}) s_{n+1}} \states^\omega) \cdot \valueof{\mdp}{s_{n+1}}
\end{align*}

Let $s_n \in \rstates$ and $s_{n+1} \in \pstates \cap \states$.
We have:
\begin{align*}
& \valueof{\mdp}{s_0} \cdot \probm_{\pmdp,s_0,\zstrat}(s_0 s_1 \cdots s_n s_{n+1} \pstates^\omega) \\
& =\ \valueof{\mdp}{s_0} \cdot \probm_{\pmdp,s_0,\zstrat}(s_0 s_1 \cdots s_n \pstates^\omega) \cdot \pprobp(s_n)(s_{n+1}) \\
& =\ \probm_{\mdp,s_0,\zstrat}(\overline{s_0 s_1 \cdots s_n} \states^\omega) \cdot \pprobp(s_n)(s_{n+1}) \cdot \valueof{\mdp}{s_n} && \text{ind.\ hyp.} \\
& =\ \probm_{\mdp,s_0,\zstrat}(\overline{s_0 s_1 \cdots s_n} \states^\omega) \cdot \probp(s_n)(s_{n+1}) \cdot \valueof{\mdp}{s_{n+1}} && \text{def.~of~$\pprobp$} \\
& =\ \probm_{\mdp,s_0,\zstrat}(\overline{s_0 s_1 \cdots s_n s_{n+1}} \states^\omega) \cdot \valueof{\mdp}{s_{n+1}}
\end{align*}
This completes the inductive step, and we have proved item~1.

Towards item~2, define an MDP $\pmdp' = \tuple{\pstates',\pzstates',\prstates',\ptransition',\pprobp'}$ with ``intermediate'' states like $(s,t)$ in~$\pmdp$, but with transition probabilities as in~$\mdp$; more precisely:
\begin{align*}
\pzstates' \ = \ &\zstates \\
\prstates' \ = \ &\rstates \cup \{(s,t) \in \mathord{\transition} \mid s \in \zstates\} \\
\ptransition' \ = \ &\{(s,(s,t)) \in (\zstates \times \mathord\transition) \mid s \transition t\}  \cup (\rstates \times \states) \cup \mbox{} \\ &\{((s,t),t) \in (\mathord\transition \times S) \mid s \in \zstates\} \\
\pprobp'(s,t) \ = \ &\probp(s,t) \\
\pprobp'((s,t),t) \ = \ & 1 
\end{align*}
Then we have 
\begin{equation}
\probm_{\pmdp',s_0,\zstrat}(\playset) \ = \ \probm_{\mdp,s_o,\zstrat}(\overline{\playset}) \qquad \text{for all measurable $\playset \subseteq s_0 (\pstates')^\omega$.}
  \label{eq-intermediate-conditioned-MDP}
\end{equation}
Let $s_0 s_1 \cdots s_n \in s_0 (\pstates')^*$.
If $s_0 s_1 \cdots s_n$ is a partial run in~$\pmdp$, then we have:
\begin{align*}
& \probm_{\pmdp',s_0,\zstrat}(s_0 s_1 \cdots s_n (\pstates')^\omega) \\
&\ = \ \probm_{\mdp,s_0,\zstrat}(\overline{s_0 s_1 \cdots s_n} \states^\omega) && \text{\cref{eq-intermediate-conditioned-MDP}}\\
&\ \ge \ \probm_{\mdp,s_0,\zstrat}(\overline{s_0 s_1 \cdots s_n} \states^\omega) \cdot \valueof{\mdp}{s_n}  \\
&\ = \ \valueof{\mdp}{s_0} \cdot \probm_{\pmdp,s_0,\zstrat}(s_0 s_1 \cdots s_n \pstates^\omega) && \text{item~1} \\
&\ = \ \valueof{\mdp}{s_0} \cdot \probm_{\pmdp,s_0,\zstrat}(s_0 s_1 \cdots s_n (\pstates')^\omega \cap \pstates^\omega)
\end{align*}
Otherwise (i.e., $s_0 s_1 \cdots s_n$ is not a partial run in~$\pmdp$), the same inequality holds trivially.
Invoking \cref{lem:measure-theory} below with $S := \pstates'$ and $s := s_0$ and $\mu(\playset) := \probm_{\pmdp,s_0,\zstrat}(\playset \cap \pstates^\omega)$ and $\mu'(\playset) := \probm_{\pmdp',s_0,\zstrat}(\playset)$ and $x := \valueof{\mdp}{s_0}$ yields
\[
\probm_{\pmdp',s_0,\zstrat}(\playset) 
\ \ge \ 
\valueof{\mdp}{s_0} \cdot \probm_{\pmdp,s_0,\zstrat}(\playset \cap \pstates^\omega) \quad \text{for all measurable } \playset \subseteq s_0 (\pstates')^\omega\,.
\]
By \cref{eq-intermediate-conditioned-MDP}, the first inequality of item~2 follows.

Towards the second inequality of item~2, define $\denotationof{\formula}{s_0}_- \eqdef \{\play \in s_0 (\pstates')^\omega \mid \overline{\play} \in \denotationof{\formula}{s_0}\}$.
If $\probm_{\pmdp',s_0,\zstrat}(s_0 s_1 \cdots s_n (\pstates')^\omega \cap \denotationof{\formula}{s_0}_-) > 0$, then $s_0 s_1 \cdots s_n$ is a partial run in~$\pmdp$ and we have:
\begin{align*}
&\valueof{\mdp}{s_0} \cdot \probm_{\pmdp,s_0,\zstrat}(s_0 s_1 \cdots s_n (\pstates')^\omega \cap \pstates^\omega) \\
& =\ \valueof{\mdp}{s_0} \cdot \probm_{\pmdp,s_0,\zstrat}(s_0 s_1 \cdots s_n \pstates^\omega) \\
& =\ \probm_{\mdp,s_0,\zstrat}(\overline{s_0 s_1 \cdots s_n} \states^\omega) \cdot\valueof{\mdp}{s_n} && \text{item~1} \\
& \ge\ \probm_{\mdp,s_0,\zstrat}(\overline{s_0 s_1 \cdots s_n} \states^\omega) \cdot \probm_{\mdp,s_0,\zstrat}(\denotationof{\formula}{s_0} \mid \overline{s_0 s_1 \cdots s_n} \states^\omega) && \text{$\formula$ is tail} \\
& =\ \probm_{\mdp,s_0,\zstrat}(\overline{s_0 s_1 \cdots s_n} \states^\omega \cap \denotationof{\formula}{s_0}) \\
& =\ \probm_{\pmdp',s_0,\zstrat}(s_0 s_1 \cdots s_n (\pstates')^\omega \cap \denotationof{\formula}{s_0}_-) && \text{\cref{eq-intermediate-conditioned-MDP}}
\end{align*}
Otherwise (i.e., $\probm_{\pmdp',s_0,\zstrat}(s_0 s_1 \cdots s_n (\pstates')^\omega \cap \denotationof{\formula}{s_0}_-) = 0$), the same inequality holds trivially.
Invoking \cref{lem:measure-theory} with $S := \pstates'$ and $s := s_0$ and $\mu(\playset) := \probm_{\pmdp',s_0,\zstrat}(\playset \cap \denotationof{\formula}{s_0}_-)$ and $\mu'(\playset) := \probm_{\pmdp,s_0,\zstrat}(\playset \cap \pstates^\omega)$ and $x := 1 / \valueof{\mdp}{s_0}$ yields
\[
\valueof{\mdp}{s_0} \cdot \probm_{\pmdp,s_0,\zstrat}(\playset \cap \pstates^\omega) 
\ \ge \ 
\probm_{\pmdp',s_0,\zstrat}(\playset \cap \denotationof{\formula}{s_0}_-) \quad \text{for all measurable } \playset \subseteq s_0 (\pstates')^\omega\,.
\]
By \cref{eq-intermediate-conditioned-MDP}, the second inequality of item~2 follows.

Item~3 follows from item~2, with $\overline{\playset} = \denotationof{\formula}{s_0}$.
\end{proof}

The following lemma was used in the preceding proof.

\begin{lemma} \label{lem:measure-theory}
Let $\states$ be countable and $s \in \states$.
Call a set of the form $s w \states^\omega$ for $w \in \states^*$ a \emph{cylinder}.
Let $\mu, \mu'$ be measures on $s S^\omega$ defined in the standard way, i.e., first on cylinders and then extended to all measurable sets $\playset \subseteq s \states^\omega$.
Suppose there is $x \ge 0$ such that $x \cdot \mu(\cyl) \le \mu'(\cyl)$ for all cylinders~$\cyl$.
Then $x \cdot \mu(\playset) \le \mu'(\playset)$ holds for all measurable $\playset \subseteq s S^\omega$.
\end{lemma}
\begin{proof}
Let $\classcyl = \{\cyl \subseteq s \states^\omega \mid \cyl \text{ cylinder}\}$ denote the class of cylinders.
This class generates an algebra $\classcyl_* \supseteq \classcyl$, which is the closure of~$\classcyl$ under finite union and complement.
The classes $\classcyl$ and $\classcyl_*$ generate the same $\sigma$-algebra $\sigma(\classcyl)$.
The class~$\classcyl_*$ is a set of countable disjoint unions of cylinders~\cite[Section~2]{billingsley-1995-probability}.
Hence $x \cdot \mu(\playset) \le \mu'(\playset)$ for all $\playset \in \classcyl_*$.

Define
\[
 \classmon = \{\playset \in \sigma(\classcyl) \mid x \cdot \mu(\playset) \le \mu'(\playset) \}\,.
\]
We have $\classcyl \subseteq \classcyl_* \subseteq \classmon \subseteq \sigma(\classcyl)$.
We show that $\classmon$ is a \emph{monotone} class, i.e., if $\playset_1, \playset_2, \ldots \in \classmon$, then $\playset_1 \subseteq \playset_2 \subseteq \cdots$ implies $\bigcup_i \playset_i \in \classmon$, and $\playset_1 \supseteq \playset_2 \supseteq \cdots$ implies $\bigcap_i \playset_i \in \classmon$.
Suppose $\playset_1, \playset_2, \ldots \in \classmon$ and $\playset_1 \subseteq \playset_2 \subseteq \cdots$.
Then:
\begin{align*}
x \cdot \mu\Big(\bigcup_i \playset_i\Big) 
& = \sup_i x \cdot \mu(\playset_i) && \text{measures are continuous from below} \\
& \le \sup_i \mu'(\playset_i) && \text{definition of~$\classmon$} \\
& = \mu'\Big(\bigcup_i \playset_i\Big) && \text{measures are continuous from below}
\end{align*}
So $\bigcup_i \playset_i \in \classmon$.
Using the fact that measures are continuous from above, one can similarly show that if $\playset_1, \playset_2, \ldots \in \classmon$ and $\playset_1 \supseteq \playset_2 \supseteq \cdots$ then $\bigcap_i \playset_i \in \classmon$.
Hence $\classmon$ is a monotone class.

Now the \emph{monotone class theorem} (see, e.g., \cite[Theorem~3.4]{billingsley-1995-probability}) implies that $\sigma(\classcyl) \subseteq \classmon$, thus $\classmon = \sigma(\classcyl)$.
Hence $x \cdot \mu(\playset) \le \mu'(\playset)$ for all $\playset \in \sigma(\classcyl)$.
\end{proof}

We prove \cref{lem:conditioned-MDP-preserves-structural-transience} from the main body:

\lemconditionedMDPpreservesstructuraltransience*

\begin{proof}
For any state $s_0 \in \pstates \cap \states$, let
\[
 \playset_{s_0} \eqdef \{s_0 s_1 \cdots \in s_0 \pstates^\omega \mid \exists\,i \ge 1: s_0 = s_i\}
\]
denote the event of returning to~$s_0$.
Suppose $\pmdp$ is not universally transient.
By \cref{lem:structural-transience}(3) there exists $s_0 \in \pstates \cap \states$ such that $\valueof{\pmdp,\playset_{s_0}}{s_0} = 1$.
We show that, in~$\mdp$, for any $C>0$ there exists a strategy under which the expected number of returns to~$s_0$ is at least~$C$.
By \cref{lem:structural-transience}(4) this implies that $\mdp$ is not universally transient.

Let $C>0$.
Let $\playset$ be the event, in~$\pmdp$, starting in~$s_0$, of returning to~$s_0$ at least $2 C / \valueof{\mdp,\formula}{s_0}$ times, and denote by $X$ the random variable counting the number of returns to~$s_0$.
Since $\valueof{\pmdp,\playset_{s_0}}{s_0} = 1$, we also have $\valueof{\pmdp,\playset}{s_0} = 1$, and so there exists a strategy~$\zstrat$ with $\probm_{\pmdp, s_0, \zstrat}(\playset) \ge \frac12$.
By the first inequality of \cref{lem:conditioned-construction}.2 we have $\probm_{\mdp,s_0,\zstrat}(\overline{\playset}) \ge \valueof{\mdp,\formula}{s_0} \cdot \frac12$.
It follows:
\[
 \expectval\textstyle_{\mdp,s_0,\zstrat}(X) \ \ge\ \probm_{\mdp,s_0,\zstrat}(\overline{\playset}) \cdot 2 C / \valueof{\mdp,\formula}{s_0} \ \ge\  C \qedhere
\]
\end{proof}

In~\cite[Lemma~6]{KMSW2017} a variant, say~$\mdp_+$, of the conditioned MDP~$\pmdp$ from \cref{def:conditionedmdp} was proposed.
This variant~$\mdp_+$ differs from~$\pmdp$ in that $\mdp_+$ has only those states~$s$ from~$\mdp$ that have an optimal strategy, i.e., a strategy~$\zstrat$ with $\probm_{\mdp,s,\zstrat}(\formula) = \valueof{\mdp}{s}$.
Further, for any transition $s\transition t$ in~$\mdp_+$ where $s$ is a controlled state, we have $\valueof{\mdp}{s} = \valueof{\mdp}{t}$, i.e., $\mdp_+$ does not have value-decreasing transitions emanating from controlled states.

As a consequence, in contrast to~$\pmdp$, in~$\mdp_+$ there is no need for intermediate states of the form $(s,t)$:
Since $\valueof{\mdp}{s} = \valueof{\mdp}{t}$, an intermediate state $(s,t)$ would transition to~$t$ with probability~$1$.
Therefore, such intermediate states do not appear in~$\mdp_+$.
Instead, in~$\mdp_+$ there is a direct transition from~$s$ to~$t$ like in the original MDP~$\mdp$.
As a further consequence, the state~$s_\bot$ does not appear in~$\mdp_+$ (it would not be reachable).

Any strategy~$\zstrat$ in~$\mdp_+$ can be naturally applied also in~$\pmdp$: whenever $\zstrat$ moves from a controlled state~$s$ to a state~$t$ (hence $s$ and~$t$ have the same value), in~$\pmdp$ strategy~$\zstrat$ moves instead to the random state $(s,t)$ (from which $\pmdp$ transitions to~$t$ with probability~$1$).

This correspondence is exploited in the proof of the following lemma from the main body:

\lemoldconditionedMDPpreservesstructuraltransience*
\begin{proof}
Suppose $\mdp$ is universally transient.
We show that $\mdp_+$ is universally transient.
Indeed, let $s_0$ be any state in~$\mdp_+$, and let $\zstrat$ be any strategy in~$\mdp_+$.
Write $\playset$ for the event of returning to~$s_0$ in~$\pmdp$, and $\overline{\playset}$ for the event of returning to~$s_0$ in~$\mdp_+$.
We have $\probm_{\mdp_+,s_0,\zstrat}(\overline{\playset}) = \probm_{\pmdp,s_0,\zstrat}(\playset)$.
Since $\pmdp$ is universally transient by \cref{lem:conditioned-MDP-preserves-structural-transience}, by \cref{lem:structural-transience}(3) this probability is less than~$1$.
Applying \cref{lem:structural-transience}(3) again, it follows that $\mdp_+$ is universally transient.
\end{proof}

\end{document}